\documentclass[12pt]{article}
\usepackage[utf8]{inputenc}
\usepackage{amsmath, amsthm, amssymb}
\usepackage{graphicx,fullpage,mathptmx,helvet}
\usepackage[dvipsnames]{xcolor}
\usepackage{hyperref,cleveref,cite,url}
\usepackage{mathtools}
\usepackage{complexity}
\usepackage{float, array}
\usepackage{relsize}
\usepackage{diagbox}
\usepackage[table]{xcolor}
\usepackage{tikz}
\usepackage{tkz-graph}
\usetikzlibrary{arrows,decorations.markings,bbox,patterns,arrows.meta}
\usetikzlibrary{shapes}

\DeclareMathOperator{\lcm}{lcm}

\newtheorem{theorem}{Theorem}
\newtheorem{question}[theorem]{Question}
\newtheorem{proposition}[theorem]{Proposition}
\newtheorem{observation}[theorem]{Observation}
\newtheorem{corollary}[theorem]{Corollary}
\newtheorem{lemma}[theorem]{Lemma}

\theoremstyle{definition}

\newcommand{\pattern}[3]{{#1}\diamond^{#2 - 1}{#3}}
\newcommand{\patternprime}[3]{{#1}\diamond^{#2}{#3}}
\newcommand{\subsequence}[2]{#1\langle#2\rangle}

\title{Pairs of square-free arithmetic progressions in infinite words}
\author{Thomas Delépine$^{1}$, Pascal Ochem$^{2}$, Matthieu Rosenfeld$^{1}$\thanks{This research was funded, in whole or in part, by the French National Research Agency (ANR) under grant agreement No.~ANR-24-CE48-3758-01. In accordance with the objective of open access dissemination, the authors apply a Creative Commons Attribution (CC-BY) license to any accepted article or manuscript (AAM) resulting from this submission.}\\[0.5em]
\small $^{1}$LIRMM, Université de Montpellier, CNRS, Montpellier, France\\
\small $^{2}$LIRMM, CNRS, Université de Montpellier, Montpellier, France}
\date{Mai 2026}

\begin{document}

\maketitle

\begin{abstract}

We study a question of Harju from 2019 regarding the existence of infinite ternary square-free words whose subsequences modulo $p$ and $q$ are also square-free for relatively prime integers $p$ and $q$. Among such pairs $(p, q)$ with $p$, $q \ge 3$, the only two pairs with this property known prior to this work were $(3, 11)$ and $(5, 6)$. We prove that there are finitely many pairs $(p, q)$ of relatively prime integers with $p, q \ge 3$ for which there is no infinite ternary square-free word whose subsequences modulo $p$ and $q$ are square-free. To prove our result, we combine different techniques, including the construction of words from multi-valued square-free morphisms and circular square-free morphisms. We also introduce the notion of square-free transducers, a generalization of square-free morphisms that may be of independent interest.
    
\end{abstract}

\section{Introduction}\label{chap:intro}
The avoidability of repetitions in infinite words is a central topic in combinatorics on words. A \emph{square} is a word of the form $uu$ with $u$ non-empty, and a word is \emph{square-free} if it does not contain any factor that is a square. Over a ternary alphabet, Thue showed the existence of infinite square-free words~\cite{thue1906uber} and since then, many variants and additional constraints have been considered.

In this context, Harju proposed in \cite{HARJU201995} the following generalization. Given an infinite word $w=w_0w_1w_2\ldots$ and an integer $p \geq 1$, the \emph{subsequence modulo $p$ of $w$}, denoted $\subsequence{w}{p}$, is defined as the sequence of letters at positions congruent to $0$ modulo $p$, that is,
\[
\subsequence{w}{p}=w_0w_pw_{2p}w_{3p}\ldots\,
\] 

We say that $w$ is \emph{square-free modulo $p$} if $\subsequence{w}{p}$ is square-free. Harju proved in \cite{HARJU201995} that for every $p\ge3$ there exists an infinite ternary square-free word that is also square-free modulo $p$. He also proved that this is impossible when $p=2$ and then asked the following question:
\begin{question}[Harju, 2019 \cite{HARJU201995}]\label{question:harju}
Do there exist pairs $(p,q)$ of relatively prime integers such that there exist infinite square-free words $w$ over the ternary alphabet such that $\subsequence{w}{p}$ and $\subsequence{w}{q}$ are both square-free?
\end{question}

Currie et al. proved in~\cite{ArithmeticProgressions} that there exist at least two such pairs $(p, q)$. For $(p,q)\in \{(3,11), (5,6)\}$, they constructed infinite ternary square-free words that are square-free modulo $p$ and $q$. They then asked for a characterization of the pairs $(p, q)$ for which such words exist. They noted that for the pair $(5,8)$ computer experiments suggest that no such words exist, but that this would be surprising since one might expect it to be easier to construct a word for this pair than for the pair $(5,6)$.

In this article, we make some progress toward the characterization of pairs $(p, q)$ of relatively prime integers for which there exist infinite ternary square-free words that are square-free modulo $p$ and $q$. We show that there exist infinite square-free words $w$ over the ternary alphabet such that $\subsequence{w}{p}$ and $\subsequence{w}{q}$ are both square-free for every pair $(p, q)$ of relatively prime integers (assuming without loss of generality that $p \le q$) such that
\begin{itemize}
\item  $p \geq 331$ and $q \geq 364$ (Theorem \ref{thm:pq_large}), or
\item $p\in \{13,17,18,19\}\cup \mathbb{N}_{\ge23}$ and $q\ge 19p$ (Corollary \ref{coro:construction_from_cyclic_morphisms}), or
\item $p \in \{3, 4, 5, 7, 8, 9, 10, 11, 12, 14, 15, 16, 20, 21, 22\}$ and $q$ large enough (Proposition \ref{prop:sparse_p_pairs}, see Table \ref{fig:morphisms_cone_remaining_cases} for the precise threshold values of $q$), or
\item $p=6$ and $q\ge 341$ (Proposition \ref{prop:p=6}).
\end{itemize}
\begin{figure}
    \centering
    \includegraphics[scale=0.45]{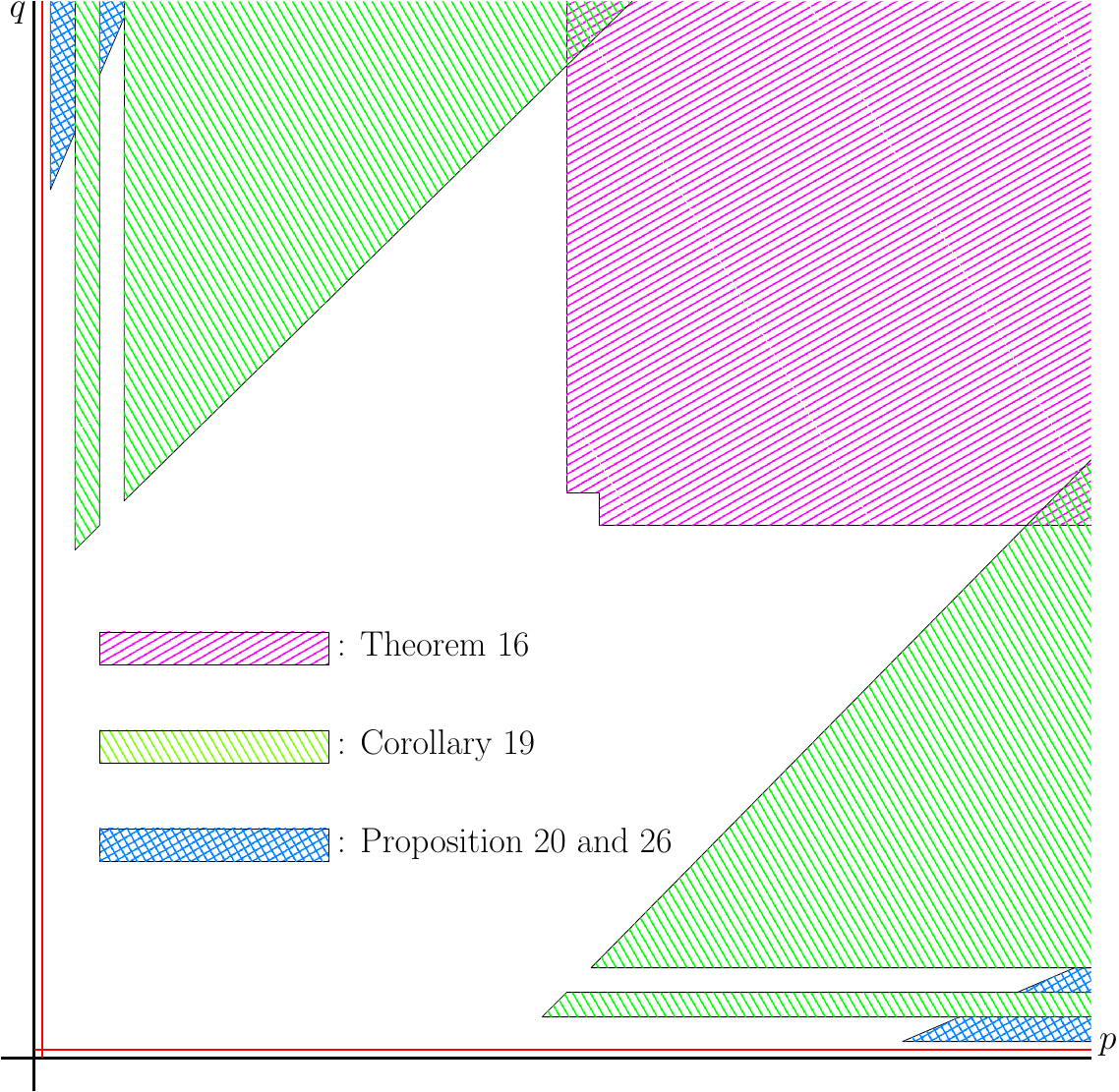}
    \caption{A (not-to-scale) schematic plot of the cases covered by Theorem \ref{thm:pq_large}, Corollary \ref{coro:construction_from_cyclic_morphisms}, Proposition \ref{prop:sparse_p_pairs} and Proposition \ref{prop:p=6}.}
    \label{fig:sketch_tiling}
\end{figure}
These four cases cover all but finitely many pairs of coprime integers $p$, $q \ge 3$, as illustrated in Figure \ref{fig:sketch_tiling}. In particular, this implies that there are only finitely many pairs of coprime integers with $p, q \ge 3$ for which it is not possible to construct such a word. We also solve the problem for all pairs with $p, q \leq 20$ other than $(5,8)$, which we leave open (see Table \ref{table:pq_leq_20}), although we conjecture that $(5, 8)$ is a negative pair. In total, we leave open $1238408$ cases, but we conjecture that almost all of them also have a positive answer.

We begin in Section \ref{sec:prel_and_palindromes} by introducing one of the main tools used throughout the paper: a particular multi-valued morphism introduced in \cite{ArithmeticProgressions} that allows us to construct square-free words while controlling the precise letters at some well-chosen positions. As an illustrative example of the method, we present an improvement on a result of Harju and Müller regarding a question of Brešar et al. on the avoidance of palindromes at prescribed positions in infinite square-free words~\cite{HARJU_palindromes,BRESAR_nonrepetitive_coloring_in_trees}.
Then, in Section \ref{chap:pqlarge}, we prove Theorem \ref{thm:pq_large}, which gives a first family of pairs of large coprime integers $(p, q)$ for which infinite ternary square-free words that are square-free modulo $p$ and $q$ exist. In Section \ref{chap:smallplargeq}, we prove Corollary \ref{coro:construction_from_cyclic_morphisms} and Proposition \ref{prop:sparse_p_pairs}. These proofs rely on the existence of circular square-free morphisms whose length is a multiple of $p$. In particular, for Corollary \ref{coro:construction_from_cyclic_morphisms}, we use morphisms constructed by Currie \cite{currie2012infinite}, and for Proposition \ref{prop:sparse_p_pairs} we provide the required morphisms. As the case $p=6$ seems to elude this technique, we cover this special case in Section \ref{sec:p_equal_6} with Proposition \ref{prop:p=6} using a different technique based on the existence of a transducer that, from any square-free word $w$, produces a square-free word $w'$ with $\subsequence{w'}{6}=w$ (so, in particular, $w'$ is square-free modulo $6$, and by controlling $w$ it also gives us some control over $w'$). In Section \ref{sec:remaining_cases}, we discuss the cases $p, q \le 20$. We end the article in Section \ref{sec:further_remarks} with a discussion of the growth rate of these languages and other related questions. Many of the proofs of the results mentioned above rely on computer verification. The programs used for these verifications are available at \cite{code}.

\section{Preliminaries with an application to forbidding palindromes}\label{sec:prel_and_palindromes}

Let $n \geq 1$ be an integer. We denote by $\Sigma_n$ the finite alphabet of size $n$ whose letters are canonically represented by the elements of $\{0, \dots, n - 1\}$. A \emph{morphism} $h$ is a function from $\Sigma_n^*$ to $\Sigma_m^*$ for some integers $n$, $m$ that is uniquely determined by its images over the source alphabet~$\Sigma_n$, and the facts that $h(\epsilon) = \epsilon$ (where $\epsilon$ denotes the empty word), and that for every non-empty word $w = aw'$, $h(w) = h(a)h(w')$. Notice that this definition can be naturally extended to infinite words. We say that a morphism $h : \Sigma_n^* \mapsto \Sigma_n^*$ is \emph{uniform} if there exists an integer $k \geq 0$ such that for every letter $a \in \Sigma_n$, $|h(a)| = k$ in which case we say that $h$ is $k$-uniform.
Let $\pi_n : \Sigma_n^* \mapsto \Sigma_n^*$ be the morphism given by $\pi_n(i) = i + 1$ whenever $i < n - 1$ and $\pi_n(n - 1) = 0$.
We say that a  morphism $h$ is \emph{circular} if for every $i \in \Sigma_n$, $$h(\pi_n(i)) = \pi_n(h(i))\,.$$
Notice that circular morphisms are uniform.
Finally, a morphism $h$ is square-free whenever for every square-free word $w$, $h(w)$ is square-free as well. 

We will use the following particularly useful morphism in our first constructions. We define $h$, a \emph{multi-valued} circular morphism over $\Sigma_3 = \{0, 1, 2\}$ as follows:
\begin{align*}
h_{23}(0)&=012102120210\mathbf{12}021201210\\
h_{24}(0)&=012102120210\mathbf{201}021201210\\
h_{25}(0)&=012102120210\mathbf{2012}021201210\\
h_{26}(0)&=012102120210\mathbf{20121}021201210
\end{align*}
and for all $i\in\{23,24,25,26\}$, $h_i(1)=\pi_3(h_i(0))$ and  $h_i(2)=\pi_3(h_i(1))$.
Given a word $w\in\Sigma_3^\omega$ and an infinite sequence $\Gamma=(\gamma_i)_{i\ge0}\in \{23,24,25,26\}^\omega$, the \emph{image of $w$ under $h$ with guiding sequence $\Gamma$} is given by 
$$h_{\Gamma}(w)=h_{\gamma_0}(w_0)h_{\gamma_1}(w_1)h_{\gamma_2}(w_2)\ldots$$
This morphism, introduced in a similar setting in \cite{ArithmeticProgressions}, will be quite handy thanks to the following property.
\begin{lemma}[Currie, Harju, Ochem, Rampersad, 2019 \cite{ArithmeticProgressions}]\label{lemma:h_is_squarefree}
For any infinite ternary square-free word $w\in\Sigma_3^\omega$, and any guiding sequence $\Gamma$, the word
$h_{\Gamma}(w)$ is square-free.
\end{lemma}
This property is extremely useful because it gives us two ways to control the square-free word $h_{\Gamma}(w)$, by changing either $w$ or $\Gamma$.

We now illustrate how to use this morphism by improving the results on a problem relating square-free words and small palindromes.
For a word $w = w_0\dots w_n$, we denote $w_n\dots w_0$, the \emph{mirror of $w$}, by $w^R$. 
A non-empty word $w$ is a \emph{palindrome} if $w = w^R$. Observe that a word $w = w_0w_1w_2$ of length $3$ is a palindrome if and only if $w_0 = w_2$. 
In 2007, Brešar, Grytczuk, Klavžar, Niwczyk and Peterin asked in \cite{BRESAR_nonrepetitive_coloring_in_trees} whether there exists an integer $N_0$ such that for every sparse set $A$ of integers with gaps at least $N_0$, there exists an infinite ternary square-free word $w$ such that for every $a \in A$, $w_aw_{a + 1}w_{a + 2}$ is not a palindrome, i.e. $w_a \neq w_{a + 2}$. 

This question was later generalized by Currie in \cite{CURRIE_palindromes} with two kinds of constraints: for each $a \in A$, you impose either $w_a = w_{a + 2}$ or you impose $w_a \neq w_{a + 2}$. Currie proved that $N_0$ exists even in this stronger setting and gave the upper bound $N_0 \leq 400$. This bound was later improved by Harju and Müller in \cite{HARJU_palindromes} down to $N_0 \leq 67$. In this section, we improve this bound down to $N_0 \leq 30$. Our proof is similar to the one used by Harju and Müller and also serves as a simple illustration of one of the main techniques used throughout this article. First, one can observe that an infinite ternary square-free word must contain palindromes and non-palindromes of size $3$ very often.

\begin{lemma}\label{lemma:palindromes_recurrence}
    For every infinite ternary square-free word $w$, and for every integer $i \geq 0$, there exists $\Delta \leq 3$ such that $w_{i + \Delta} = w_{i + \Delta + 2}$ and $\Delta'\le1$ such that $w_{i + \Delta'} \neq w_{i + \Delta' + 2}$.
\end{lemma}

\begin{proof}
    To prove this result, it is enough to consider every square-free word of length $6$ and to assume that $i = 0$. Moreover, since letters play symmetric roles, we can assume without loss of generality that the prefix is $01$. In what follows, we highlight the beginning of a palindrome by $\hat{a}$ and the beginning of a non-palindrome by $\bar{a}$. 
    \[\hat0\bar10201\quad\hat0\bar10210\quad\hat0\bar10212\quad\bar012\hat010\quad\bar01\hat2021\quad\bar0\hat12101\quad\bar0\hat12102\qedhere\]
\end{proof}

In other words, any set of $4$ consecutive positions of an infinite ternary square-free word must contain the beginning of a factor $abc$ and the beginning of a factor $aba$. Assume that we managed to construct an infinite ternary square-free word with a prefix respecting the conditions. The following Lemma shows how to increase the size of said prefix.

\begin{lemma}\label{lemma:palindromes_construction}
    Let $t$ be an infinite ternary square-free word, let $\Gamma = (\gamma_i)_{i \geq 0}$ be a guiding sequence, and let $N  \geq 0$. Then for every integer $N' \geq N + 30$ and for every $c \in \{\top, \bot\}$, there exists a guiding sequence $\Gamma'$ such that
    \begin{itemize}
        \item for every integer $i \leq N$, $h_{\Gamma'}(t)_i=h_\Gamma(t)_i$,
        \item  and $h_{\Gamma'}(t)_{N'} = h_{\Gamma'}(t)_{N' + 2}$ if and only if $c = \top$. 
    \end{itemize}
\end{lemma}
\begin{figure}
        \centering
        \includegraphics[scale=0.9]{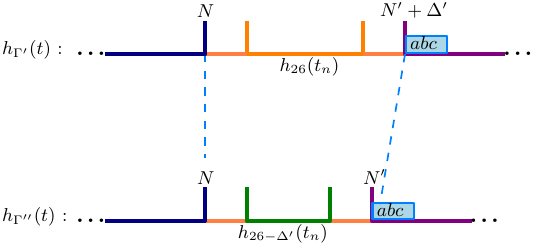}
        \caption{Illustration of the main idea behind Lemma \ref{lemma:palindromes_construction}. In order to have $abc$ at position $N'$ with either $a = c$ or $a \neq c$, we first find it at position $N' + \Delta'$ (with $\Delta'\le 3$) and then we transform an image of $h_{26}$ into either an image of $h_{23}$, $h_{24}$, $h_{25}$ or $h_{26}$ to shift $abc$ towards position $N'$.}
        \label{fig:simple_contraction_example_for_palindromes}
\end{figure}

The proof of Lemma \ref{lemma:palindromes_construction} takes advantage of the fact that $h$ is a multi-valued morphism with images of length $23$, $24$, $25$ and $26$ for each letter in $\Sigma_3$. Assume that we want a palindrome of size three to appear at position $N'$ in $h_\Gamma(t)$ for some infinite ternary square-free word $t$. We know that there exists $\Delta'\le3$ such that a palindrome appears at position $N' + \Delta'$ in $h_\Gamma(t)$. If $N' - N$ is sufficiently large, we can find a full image of $h_{26}(t_n)$ for some $n \geq 0$ between the constraint at position $N$ and the constraint at position $N'$, and by changing it to $h_{26 - \Delta'}(t_n)$ we move the palindrome to the desired position, as illustrated in Figure \ref{fig:simple_contraction_example_for_palindromes}, while preserving the prefix.

\begin{proof}
    Let $n$ be the smallest integer such that $N \leq 11 + \sum_{i = 0}^{n - 1}\gamma_i$. With this choice, $n$ is the smallest integer such that for every $i \geq n$, we can modify $\gamma_i$ without changing the prefix of size $N + 1$ of $h_\Gamma(t)$. Indeed, for every $a \in \Sigma_3$, every image of $a$ under $h$ shares a common prefix of length 12. Observe that then $N + 14 \geq \sum_{i = 0}^{n - 1}\gamma_i$. Indeed, by minimality of $n$,
    \begin{equation*}
    N+14 > 14+11 + \sum_{i = 0}^{n - 2} \gamma_i 
    \ge 26 + \sum_{i = 0}^{n - 2} \gamma_i\,
    \geq \sum_{i = 0}^{n - 1} \gamma_i.
    \end{equation*}
    We start by constructing a new guiding sequence $\Gamma' = (\gamma'_i)_{i \geq 0}$. For every integer $i \leq n - 1$, $\gamma'_i = \gamma_i$ and for every integer $i \geq n$, $\gamma'_i = 26$. Let $\Delta'$ be the smallest non-negative integer such that $h_{\Gamma'}(t)_{N' + \Delta'}h_{\Gamma'}(t)_{N' + \Delta' + 1}h_{\Gamma'}(t)_{N' + \Delta' + 2}$ is a palindrome if and only if $c = \top$. By Lemma \ref{lemma:palindromes_recurrence}, $\Delta' \leq 3$. If $\Delta' = 0$, we are done. Therefore, for the rest of the proof, we can assume that $\Delta' \geq 1$. We finally construct $\Gamma'' = (\gamma''_i)_{i \ge 0}$ from $\Gamma'$ such that $\gamma''_n = \gamma'_n - \Delta'$ and $\gamma''_i = \gamma'_i$ for all $i\neq n$. 
    
     For every $i < \sum_{j = 0}^{n - 1} \gamma_j$, $h_{\Gamma''}(t)_i = h_\Gamma(t)_i$ since for every $i \leq n - 1$, $\gamma''_i = \gamma'_i=\gamma_i$. Moreover, for every $i$ such that $\sum_{j = 0}^{n - 1} \gamma_j\le i\le N$, we have $h_{\Gamma''}(t)_i = h_\Gamma(t)_i$ since $N \leq 11 + \sum_{j = 0}^{n - 1} \gamma_j$, and for every $a \in \Sigma_3$, every image of $a$ under $h$ shares a common prefix of length 12. We then have,
    \begin{align*}
        16 + \sum_{i = 0}^{n - 1} \gamma''_i    &\leq 16 + N + 14\\
                                                &\leq N' \textrm{ since by hypothesis, } N' \geq N + 30.
    \end{align*}
    So in particular, $17 + \sum_{i = 0}^{n - 1} \gamma''_i \le N' + \Delta'$. For every $a \in \Sigma_3$, every image of $a$ under $h$ shares a common suffix of length $9 = 26 - 17$, so for every integer $i \geq N' + \Delta'$, $i \geq 17 + \sum_{j = 0}^{n - 1}\gamma'_j$ therefore $h_{\Gamma''}(t)_{i - \Delta'} = h_{\Gamma'}(t)_{i}$ so $h_{\Gamma''}(t)_{N'} = h_{\Gamma''}(t)_{N' + 2}$ if and only if $c = \top$ as desired.
\end{proof}

We can finally describe the construction procedure with the following Proposition.
\begin{proposition}\label{prop:palindromes}
    Let $(p_i)_{i \geq 0}$ be an increasing sequence of integers and let $(c_i)_{i \geq 0}$ be a sequence of elements over $\{\top, \bot\}$. If for every $i \geq 0$, $p_{i + 1} - p_i \geq 30$, then there exists an infinite ternary square-free word $w$ such that for every $i \geq 0$, $w_{p_i}w_{p_i + 1}w_{p_i + 2}$ is a palindrome if and only if $c_i = \top$.
\end{proposition}

\begin{proof}
Let $t$ be an arbitrary infinite ternary square-free word, let $\Gamma = (\gamma_i)_{i \geq 0}$ be a guiding sequence, and let $\delta\ge p_0$ be such that $h_{\Gamma}(t)_\delta = h_{\Gamma}(t)_{\delta + 2}$ if and only if $c_0 = \top$. 
We will construct from $\Gamma$ a new guiding sequence $\Gamma'$ such that for every $i \geq 0$, $h_{\Gamma'}(t)_{p_i + \delta} = h_{\Gamma'}(t)_{p_i + \delta + 2}$ if and only if $c_i = \top$. We then just need to delete the prefix of size $\delta - p_0$ of $h_{\Gamma'}(t)$ to conclude.    

Let $n$ be the largest integer such that for every $i \in \{0, \dots, n\}$, $h_\Gamma(t)_{p_i + \delta} = h_\Gamma(t)_{p_i + \delta + 2}$ if and only if $c_i = \top$. In other words, $h_\Gamma(t)$ satisfies the $n+1$ first constraints. Applying Lemma \ref{lemma:palindromes_construction} with $t$, $\Gamma$, $N = p_n + \delta$, $N' = p_{n + 1} + \delta$ and $c = c_{n + 1}$, we increase the value of $n$ (by at least $+ 1$), and we indeed have that $N' \ge N + 30$ since for every $i \ge 0$, $p_{i + 1}-  p_i \ge 30$ by hypothesis. By applying this inductively, we can ensure that an arbitrarily long prefix of the word has the desired property.
\end{proof}

\section{Large $p$ and large $q$}\label{chap:pqlarge}

Intuitively, when $p$ and $q$ are large, enforcing that the subsequences modulo $p$ and $q$ are square-free results in sparse local constraints, therefore one might expect the existence of infinite ternary square-free words that are square-free modulo $p$ and square-free modulo $q$. 
In this section, we confirm this intuition.

\subsection{Necessary and sufficient condition for large $p$ and $q$}

A \emph{pattern} is a set of words over some alphabet $\Sigma$. In particular, we denote by 
$$\mathcal{P} = \patternprime{A_0}{\delta_0}{A_1} \dots \patternprime{A_{n - 1}}{\delta_{n-1}}{A_n}$$ 
the pattern containing words of length $n + 1 + \sum_{i = 0}^{n-1}\delta_i$ over $\Sigma$ such that $w \in \mathcal{P}$ if and only if for every integer $m \le n$, $w_{m+\sum_{i = 0}^{m-1}\delta_i}$ is in the alphabet $A_m$. For the sake of conciseness, in pattern notation, we denote $\{a\}$ by $a$, and $\Sigma \setminus\{ a\}$ by $\overline{a}$ when $\Sigma$ is clear from context. We say that a word $w=w_0w_1\ldots$ \emph{contains} the pattern $\mathcal{P}$ at position $i$ (or  $\mathcal{P}$ \emph{occurs} at position $i$ in $w$) whenever there exists $v \in \mathcal{P}$, $w_i\dots w_{i + |v| - 1} = v$. For instance, $\mathcal{P} = 0\diamond^20$ over $\Sigma_2$
is exactly $\{0000, 0010, 0100, 0110\}$, and the word $010010$ contains the pattern $\mathcal{P}$ at position $0$ and at position $2$.
A pattern $\mathcal{P}$ is \emph{$(\Delta, h_{26})$-recurrent} if for every infinite ternary square-free word $t$, and for every integer $i \ge 0$, there exists $j$ such that $h_{26}(t)$ contains $\mathcal{P}$ at position $j$ with $i \le j \le i + \Delta$.
For instance, if $\mathcal{P} = \{0\}$, then $\mathcal{P}$~is $(3, h_{26})$-recurrent since every factor of length $4$ of a square-free word over $\Sigma_3$ must contain $0$, $1$, and $2$.
The idea behind the definition of this notion is that if we need a recurrent pattern at a specific position, we can always find it a bit further and then shift it towards the left by ``contracting the images of the morphism". This is what we did in the proof of Lemma \ref{lemma:palindromes_construction} where we implicitly used the fact that palindromes of size 3 are $(3, h_{26})$-recurrent.

For an infinite ternary word $w$ and a pair $(p, q)$ of integers, we denote by (\ref{condAdjacent}) the following condition :
\begin{equation}\label{condAdjacent}
\forall i \ge 0,\ 
\begin{cases}
i \equiv 0 \pmod{p} \text{ and } i+1 \equiv 0 \pmod{q}, \\
\text{or} \\
i \equiv 0 \pmod{q} \text{ and } i+1 \equiv 0 \pmod{p}
\end{cases}
\Rightarrow w_i \neq w_{i+1}\,. \tag{\textasteriskcentered}
\end{equation}

Let $p$, $q \in \mathbb{N}$. If there exists an infinite ternary square-free word $w$ that is square-free modulo $p$ and square-free modulo $q$, then in particular it satisfies the condition (\ref{condAdjacent}). In this section, we shall prove the following theorem stating that when $p$ and $q$ are large, a partial converse holds.

\begin{theorem}\label{thm:sufficient_condition}
    Let $p, q \in \mathbb{N}_{\ge 331}$ and let $w$ be an infinite ternary (not necessarily square-free) word satisfying condition (\ref{condAdjacent}). 
    Then there exists an infinite ternary square-free word $w'$ such that, $\subsequence{w'}{p} = \subsequence{w}{p}$ and $\subsequence{w'}{q} = \subsequence{w}{q}$.
\end{theorem}

So, in particular, whenever $\subsequence{w}{p}$ and $\subsequence{w}{q}$ are square-free, $w'$ is square-free, square-free modulo $p$ and square-free modulo $q$. The rest of this subsection is dedicated to the proof of Theorem \ref{thm:sufficient_condition}. Then, in Subsection \ref{sec:pairs_that_respect_the_condition}, we will study for which pairs $(p, q)$ of relatively prime integers, there exists a word $w$ that is square-free modulo $p$, square-free modulo $q$ and that respects (\ref{condAdjacent}). 

To prove Theorem \ref{thm:sufficient_condition}, we start from $h_{\Gamma}(t)$ where $t$ is any infinite ternary square-free word and $\Gamma$ is the guiding sequence of constant value $26$, and we will modify $\Gamma$ so that $\subsequence{h_\Gamma(t)}{p} = \subsequence{w}{p}$ and $\subsequence{h_\Gamma(t)}{q} = \subsequence{w}{q}$. With this technique, the subsequences modulo $p$ and $q$ are fixed and can be seen as additional constraints asking for some letters at some positions. Since $p$ and $q$ are large, these constraints are sparse, but sometimes, when the subsequences modulo $p$ and $q$ almost coincide, these constraints are actually patterns of shape $\patternprime{a}{\delta}{b}$ with $\delta$ small. In what follows, we will prove some lemmas that will help us deal with such pairs of constraints.

\begin{lemma}\label{lemma:contraction_for_Delta_recurrent_patterns}
    Let $t$ be an infinite ternary square-free word, let $\Gamma = (\gamma_i)_{i \ge 0}$ be a guiding sequence, let $\mathcal{P}$ be a $(\Delta, h_{26})$-recurrent pattern, and let $N \ge 0$. Then for every integer $N' \ge N + 4 + 26\lceil \Delta/3\rceil$, there exists a guiding sequence $\Gamma'$ such that
    \begin{itemize}
        \item for every integer $i \le N$, $h_{\Gamma'}(t)_i = h_\Gamma(t)_i$, and
        \item $\mathcal{P}$ occurs at position $N'$ in $h_{\Gamma'}(t)$.
    \end{itemize} 
\end{lemma}

In the proof of this lemma, we use the fact that the multi-valued morphism $h$ has images of length $23$, $24$, $25$, $26$ for each letter in $\Sigma_3$. The idea is that by taking $N' - N$ sufficiently large, we will be able to contract enough images of $h_{26}$ into images of $h_{23}$, between position $N$ and position $N'$, to shift $\mathcal{P}$ towards the left at position $N'$, as illustrated in Figure \ref{fig:simple_contraction_example}. Each such contracted image can shift the position of $\mathcal{P}$ by $3$ towards the left, so we need to contract roughly $\Delta/3+O(1)$ images, and if the gap between two consecutive constraints is larger than $26\Delta/3 + O(1)$, we can ensure that we have enough images of $h_{26}$ to contract without changing the prefix of size $N$. Controlling the precise value of the $O(1)$ term is the only technical detail remaining.
\begin{figure}
        \centering
        \includegraphics[scale=0.9]{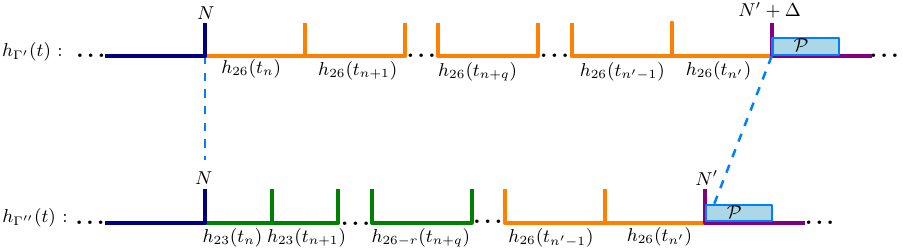}
        \caption{In order to have the pattern $\mathcal{P}$ at position $N'$, we first find it at position $N' + \Delta$ and then we transform some images of $h_{26}$ into images of $h_{23}$, $h_{24}$, $h_{25}$ or $h_{26}$ to shift $\mathcal{P}$ towards the left up to position $N'$.}
        \label{fig:simple_contraction_example}
\end{figure}
\begin{proof}
    Let $n$ be the smallest integer such that $N \le 11 + \sum_{i = 0}^{n - 1} \gamma_i$. Observe that then $N + 14 \ge \sum_{i = 0}^{n - 1} \gamma_i$.
    Indeed, by minimality of $n$, 
    \begin{equation*}
        N+14 > 14+11 + \sum_{i = 0}^{n - 2} \gamma_i    \ge 26 + \sum_{i = 0}^{n - 2} \gamma_i            
        \ge \sum_{i = 0}^{n - 1} \gamma_i\,.
    \end{equation*}
    
    We let $\Gamma' = (\gamma'_i)_{i \ge 0}$ be the guiding sequence such that for every integer $i \le n - 1$, $\gamma'_i = \gamma_i$ and for every integer $i \ge n$, $\gamma'_i = 26$. 
    Let $\Delta'$ be the smallest non-negative integer such that $h_{\Gamma'}(t)$ contains the pattern $\mathcal{P}$ at position $N' + \Delta'$, so by hypothesis, $\Delta' \le \Delta$ since $\mathcal{P}$ is $(\Delta, h_{26})$-recurrent. If $\Delta' = 0$, then $h_{\Gamma'}(t)$ already contains $\mathcal{P}$ at position $N'$ and we are done. Therefore, for the rest of the proof, we can assume that $\Delta' \ge 1$. Let $q$ and $r$ be the quotient and the remainder in the Euclidean division of $\Delta'$ by $3$, so $\Delta' = 3q + r$. We now let $\Gamma''= (\gamma''_i)_{i \ge 0}$ be the guiding sequence such that $\gamma''_i = \gamma'_i$ except for $\gamma''_n = \dots = \gamma''_{n + q - 1} = 23$ and $\gamma''_{n + q} = 26 - r$. 
    
    For every $i < \sum_{j = 0}^{n - 1}\gamma_j$, $h_{\Gamma''}(t)_i = h_\Gamma(t)_i$ since for every $i \le n - 1$, $\gamma''_i = \gamma_i$. Moreover, for every $i\in\{\sum_{j = 0}^{n - 1}\gamma_j, \dots, N\}$, $h_{\Gamma''}(t)_i = h_\Gamma(t)_i$ since $N \le 11 + \sum_{j = 0}^{n - 1}\gamma_j$, and for every $a \in \Sigma_3$, every image of $a$ under $h$ shares a common prefix of length $12$. 
    
    Observe that when constructing $\Gamma''$ from $\Gamma'$, we have only changed the values of $\gamma'_n, \dots, \gamma'_{n + \lceil\Delta'/3\rceil - 1}$ since when $\Delta' \equiv 0 \pmod3$, $\gamma''_{n + q} = 26 = \gamma'_{n + q}$. We thus have,
    \begin{align*}
        16 + \sum_{i = 0}^{n + \lceil\Delta'/3\rceil - 2} \gamma'_i &= 16 + \sum_{i = 0}^{n - 1} \gamma'_i + \sum_{i = n}^{n + \lceil\Delta'/3\rceil - 2} \gamma'_i\\
        &\le 16 + N + 14 + 26(\lceil\Delta'/3\rceil - 1)\\
        &= 4 + N + 26\lceil\Delta'/3\rceil\\
        &\le 4 + N + 26\lceil\Delta/3\rceil\textrm{ since } \Delta' \le \Delta\\
        &\le N' \textrm{ by hypothesis}.
    \end{align*}
    So, in particular, $17 + \sum_{i = 0}^{n + \lceil\Delta'/3\rceil - 2} \gamma'_i \le N' + \Delta'$.
    For every $a \in \Sigma_3$, every image of $a$ under $h$ shares a common suffix of length $9 = 26 - 17$, so for every integer $i \ge N' + \Delta'$, $i \ge 17 + \sum_{j = 0}^{n + \lceil\Delta'/3\rceil - 2} \gamma'_j$ therefore $h_{\Gamma''}(t)_{i - \Delta'} = h_{\Gamma'}(t)_{i}$ so $h_{\Gamma''}(t)$ contains the pattern $\mathcal{P}$ at position $N'$ as desired. 
\end{proof}

The following lemma lists some recurrent patterns that we will later use with Lemma \ref{lemma:contraction_for_Delta_recurrent_patterns}.

\begin{lemma}\label{lemma:bound_28_to_see_pattern_again}
    For every letter $a$, $b \in \Sigma_3$ such that $a \neq b$, the pattern $ab$ is $(12, h_{26})$-recurrent.
    Moreover, for every $\delta \in \{1, 2, 4, 5, 6, 7, 8\}$ and for every letter $a$, $b \in \Sigma_3$, the pattern  $\patternprime{a}{\delta}{b}$ is $(27, h_{26})$-recurrent.  
\end{lemma}

In order to prove Lemma \ref{lemma:bound_28_to_see_pattern_again}, we perform an exhaustive verification. We generate all possible factors of length 40 of $h_{26}(t)$ for every square-free word $t$, and we check that each pattern occurs at position at most $12$ or $27$ in every generated factor. The program computing this verification can be found in \cite[lemma\_\ref{lemma:bound_28_to_see_pattern_again}.cpp]{code}.

In what follows, we need to be able to deal with patterns of the form $\patternprime{a}{\delta}{b}$ that are not covered by Lemma \ref{lemma:bound_28_to_see_pattern_again}.
In particular, since $h_{26}$ is $26$-uniform and circular, the pattern $\patternprime{a}{25}{a}$ never appears in any image of a square-free word by $h_{26}$. Some of the other patterns that we will use are $(\Delta, h_{26})$-recurrent, but for some very large value of $\Delta$. To solve these problems, and to improve the bounds that would be worsened by these patterns, we use the following generalizations of the notion of $(\Delta, h_{26})$-recurrence and of Lemma~\ref{lemma:contraction_for_Delta_recurrent_patterns}.

A pattern $\mathcal{P}$ is \emph{$(\Delta, h)$-constructible} if there exists $k \in \mathbb{N}$ such that for every square-free word $w$ of length $k$, there exists a guiding sequence $\Gamma$ of length $k$ such that $h_\Gamma(w)$ contains $\mathcal{P}$ at position at most $\Delta$. Notice that every $(\Delta, h_{26})$-recurrent pattern is $(\Delta, h)$-constructible. With this notion, we can now first modify slightly the guiding sequence to construct the pattern at some position, and we will then, as previously, contract some images of $h_{26}$ into shorter images to shift the pattern towards the left at the desired position.
This definition leads to the following lemma. 
\begin{lemma}\label{lemma:contraction_for_Delta_constructible_patterns}
    Let $t$ be an infinite ternary square-free word, let $\Gamma = (\gamma_i)_{i \ge 0}$ be a guiding sequence and let $\Delta \geq 0$. For every integer $N \ge 0$, $N' \ge N + 26\lceil(\Delta + 1)/3\rceil + 198$, and for every $(\Delta, h)$-constructible pattern $\mathcal{P}$,  there exists a guiding sequence $\Gamma'$ such that
    \begin{itemize}
        \item for every integer $i \le N$, $h_\Gamma(t)_i = h_{\Gamma'}(t)_i$, and
        \item $h_{\Gamma'}(t)$ contains an occurrence of the pattern $\mathcal{P}$ at position $N'$.
    \end{itemize}
\end{lemma}

The proof of Lemma \ref{lemma:contraction_for_Delta_constructible_patterns} proceeds as for Lemma \ref{lemma:contraction_for_Delta_recurrent_patterns} with one additional step that aims at constructing the pattern.

\begin{figure}
    \centering
    \includegraphics[scale=0.9]{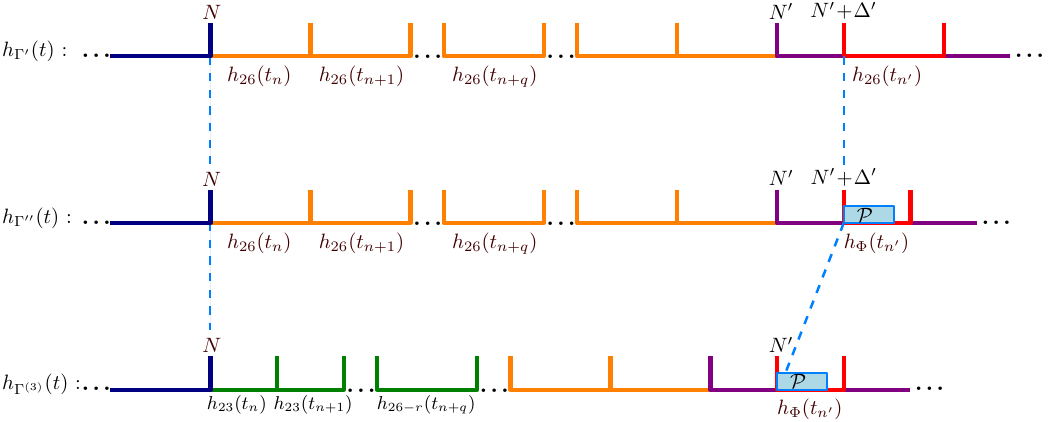}
    \caption{This figure illustrates the proof of Lemma \ref{lemma:contraction_for_Delta_constructible_patterns}. In order to have the pattern $\mathcal{P}$ at position $N'$, we first construct it at position $N' + \Delta'$ and then we transform some images of $h_{26}$ into images of $h_{23}$, $h_{24}$, $h_{25}$ or $h_{26}$ to shift $\mathcal{P}$ towards the left up to position $N'$.}
    \label{fig:pattern_a4b_proof}
\end{figure}

\begin{proof}
    Let $n$ be the smallest integer such that $N\le 11+\sum_{i=0}^{n - 1}\gamma_i$, so $N + 14 \ge \sum_{i=0}^{n - 1}\gamma_i$.
    We first let $\Gamma' = (\gamma'_i)_{i\ge 0}$ be the guiding sequence such that for every integer $i \le n - 1$, $\gamma'_i = \gamma_i$ and for every integer $i \ge n$, $\gamma'_i = 26$. Let $n'$ be the smallest integer such that $N' \le  \sum_{i = 0}^{n' - 1}\gamma'_i$, then $N' + 25  \ge \sum_{i = 0}^{n' - 1}\gamma'_i$. Indeed, by minimality of $n'$, $N' >  \sum_{i = 0}^{n' - 2}\gamma'_i$ so $N' +26 >  \sum_{i = 0}^{n' - 2}\gamma'_i + 26 = \sum_{i = 0}^{n' - 1}\gamma'_i$.
    
    Since $\mathcal{P}$ is $(\Delta, h)$-constructible, there exists a finite guiding sequence $\Phi = (\varphi_i)_{i \ge 0}$ of length $k + 1$ for some $k \geq 0$ such that $h_\Phi(t_{n'}\dots t_{n' + k})$ contains $\mathcal{P}$ at position $\Delta' \le \Delta$. We define a new guiding sequence $\Gamma'' = (\gamma''_i)_{i \ge 0}$ from $\Gamma'$ and $\Phi$ by setting $\gamma''_i = \gamma'_i$ for every $i \ge 0$ except for $i \in \{0, \dots, k\}$ where $\gamma''_{n' + i} = \varphi_i$.
    Hence, $h_{\Gamma''}(t)$ contains $\mathcal{P}$ at position 
    \begin{align*}
        \Delta' + \sum_{i=0}^{n'-1}\gamma''_i &= \Delta' + \sum_{i=0}^{n' - 1}\gamma'_i \textrm{ since for all } i \le n' - 1, \gamma''_i = \gamma'_i\\
        &\le\Delta' + N' + 25 \textrm{ since } \sum_{i=0}^{n'-1}\gamma'_i \le N' + 25\\
        &\le N' + 25 + \Delta.
    \end{align*}

     Since $\Delta' + \sum_{i=0}^{n'-1}\gamma''_i \geq N'$, let $\Delta'' \geq 0$ be the smallest integer such that $h_{\Gamma''}(t)$ contains $\mathcal{P}$ at position $N' + \Delta''$. By the previous observation, $\Delta'' \le 25 + \Delta$. If $\Delta'' = 0$, then $h_{\Gamma''}(t)$ already contains $\mathcal{P}$ at position $N'$ as desired, so we can assume that $\Delta'' \ge 1$. Let $q$ and $r$ be the quotient and the remainder in the Euclidean division of $\Delta''$ by $3$, so $\Delta'' = 3q + r$. We now construct $\Gamma^{(3)} = (\gamma^{(3)}_i)_{i \ge 0}$ from $\Gamma''$ by setting $\gamma^{(3)}_i = \gamma''_i$ except for $\gamma^{(3)}_n, \dots, \gamma^{(3)}_{n + q - 1} = 23$ and $\gamma^{(3)}_{n + q} = 26 - r$. For every $i \le \sum_{j = 0}^{n - 1}\gamma_j$, we have that $h_{\Gamma^{(3)}}(t)_i = h_\Gamma(t)_i$ since $N \le 11 + \sum_{j = 0}^{n - 1}\gamma_j$ and for every $a \in \Sigma_3$, every image of $a$ under $h$ shares a common prefix of length $12$. Observe that when constructing $\Gamma^{(3)}$ from $\Gamma''$, we have only changed the values of $\gamma''_n, \dots, \gamma''_{n + \lceil \Delta''/3\rceil - 1}$ since when $\Delta'' \equiv 0\pmod3$, $\gamma^{(3)}_{n + q} = 26 = \gamma''_{n + q}$. We thus have
    \begin{align*}
        16 + \sum_{i = 0}^{n + \lceil \Delta''/3\rceil - 2}\gamma''_i &= 16 + \sum_{i = 0}^{n - 1}\gamma''_i + \sum_{i = n}^{n + \lceil \Delta''/3\rceil - 2}\gamma''_i\\
        &\le 16 + N + 26(\lceil\Delta''/3\rceil - 1)\\
        &= N + 26\lceil\Delta''/3\rceil -10\\
        &\le N + 26\lceil (\Delta + 25)/3 \rceil -10 \textrm{ since } \Delta'' \leq \Delta + 25\\
        &=N + 26\lceil(\Delta + 1)/3\rceil + 198\\
        &\le N' \textrm{ since by hypothesis, } N' - N \ge 26\lceil(\Delta + 1)/3 \rceil + 198.
    \end{align*}
So, in particular, $17 + \sum_{i=0}^{n + \lceil\Delta/3\rceil-2}\gamma''_i \le N' + \Delta''$. For every $a \in \Sigma_3$, every image of $a$ under $h$ shares a common suffix of length $9 = 26 - 17$, so for every integer $\ell \ge N' + \Delta''$, we have $\ell\ge 17 + \sum_{i=0}^{n + \lceil\Delta/3\rceil-2}\gamma''_i$ therefore $h_{\Gamma^{(3)}}(t)_{\ell - \Delta''} = h_{\Gamma''}(t)_\ell$ so $h_{\Gamma^{(3)}}(t)$ contains the pattern $\mathcal{P}$ at position $N'$ as desired.
\end{proof}

The following Lemma provides some constructability bounds for patterns that will be useful in the proof of Theorem \ref{thm:sufficient_condition}.

\begin{lemma}\label{lemma:Delta_constructability_of_A3B_and_Ageq9B_bis}
    For every $a$, $b \in \Sigma_3$,
    \begin{itemize}
        \item the pattern $\patternprime{a}{3}{b}$ is $(10, h)$-constructible, and
        \item for every integer $\delta \ge 9$, the pattern $\patternprime{a}{\delta}{b}$ is $(8, h)$-constructible.
    \end{itemize}
\end{lemma}

\begin{proof}
For patterns $\patternprime{a}{3}{b}$ and $\patternprime{a}{\delta}{b}$ with $\delta \in \{9, \dots, 15\}$ an exhaustive verification is provided in \cite[lemma\_\ref{lemma:Delta_constructability_of_A3B_and_Ageq9B_bis}.cpp]{code}. For each pattern $\mathcal{P}$ and each square-free pre-image $t$ of length $2$, we compute by exhaustive search a guiding sequence $\Gamma_{\mathcal{P}, t}$ minimizing the position of the first occurrence of $\mathcal{P}$ in $h_{\Gamma_{\mathcal{P}, t}}(t)$, and we check that this position is at most $10$ (resp. $8$).

For the case of $\delta \geq 16$, we will actually prove $(2, h)$-constructability by simple inspection of $h$. Let $t$ be a finite square-free word whose image under $h$ should contain the pattern $\patternprime{a}{\delta}{b}$ at position at most $2$ and let $\ell \le 2$ be such that $h_{26}(t_0)_\ell = a$. Let $\Gamma = (\gamma_i)_{i \in \{0,\dots, |t| - 1\}}$ be a guiding sequence with $\gamma_0 = 26$. Since $h$ is a square-free morphism and $t$ is square-free, there exists $\Delta \le 3$ such that $h_\Gamma(t)_{\ell + \delta + \Delta} = b$. By setting $\gamma_0 = 26 - \Delta$ and since $\ell + \delta \ge 17$, $h_\Gamma(t)$ contains the pattern $\patternprime{a}{\delta}{b}$ at position $\ell \le 2$ as desired.
\end{proof}

We are now ready to prove Theorem \ref{thm:sufficient_condition}.
\begin{proof}[Proof of Theorem \ref{thm:sufficient_condition}]

Let $w$ be an infinite ternary word respecting condition (\ref{condAdjacent}). We want to construct an infinite ternary square-free word $w'$ such that $\subsequence{w'}{p} =\subsequence{w}{p}$ and $\subsequence{w'}{q} = \subsequence{w}{q}$. Let $(u_i)_{i \ge 0}$ be the ordered sequence of the multiples of $p$ or of $q$. So, for every $n \ge 0$, $w'_{u_n}$ must be equal to $w_{u_n}$. We start by setting $w'$ to be $h_{26}(t)$ for some infinite ternary square-free word $t$ such that $w'_0 = w_0$. 

Let $n \ge 0$ be such that for every integer $i \le n$, $w'_{u_i}  = w_{u_i}$ and the prefix of length $u_n + 1$ of $w'$ is square-free, and we will show how to increase the value of $n$ by only doing modifications on $\Gamma$ so that $w'$ remains square-free.
\begin{itemize}
    \item If $u_{n + 1} - u_{n} \ge 30$, then since the pattern $w_{u_{n + 1}}$ is $(3, h_{26})$-recurrent (every ternary square-free word of length $4$ must contain the three letters), we can apply Lemma \ref{lemma:contraction_for_Delta_recurrent_patterns} with $P = w_{u_{n + 1}}$, $\Delta \le 3$, $N = u_n$, $N' = u_{n + 1}$, and we indeed have that $N' \ge N + 4 + 26\lceil3/3\rceil$ since by hypothesis, $N' \ge N + 30$ and $4 + 26\lceil\Delta/3\rceil \le 30$.

    \item Otherwise, if $u_{n + 1} - u_{n} \le 29$, then $u_n - u_{n - 1} \ge 302$ since $u_{n + 1} - u_{n - 1} \ge \min(p, q) \ge 331$. Let $\delta = u_{n + 1} - u_n - 1$, we would like to find or construct the pattern $\patternprime{w_{u_n}}{\delta}{w_{u_{n + 1}}}$ and shift it towards the left. 
    \begin{itemize}
        \item If $\delta = 0$, then in particular $w_{u_n}\neq w_{u_{n + 1}}$ and therefore the pattern $\patternprime{w_{u_n}}{\delta}{w_{u_{n + 1}}}$ is $(12, h_{26})$-recurrent by Lemma \ref{lemma:bound_28_to_see_pattern_again}, thus we can apply Lemma \ref{lemma:contraction_for_Delta_recurrent_patterns} with $\mathcal{P} = \patternprime{w_{u_n}}{\delta}{w_{u_{n + 1}}}$, $\Delta \leq 12$, $N = u_{n - 1}$, $N' = u_{n}$ and we indeed have that $N' \ge N + 4 + 26\lceil\Delta/3\rceil$ since by hypothesis $N' \ge N + 302>N+108\ge N+4 + 26\lceil\Delta/3\rceil$.
        \item If $\delta \in \{1, 2, 4, 5, 6, 7, 8\}$, then by Lemma \ref{lemma:bound_28_to_see_pattern_again}, the pattern $\patternprime{w_{u_n}}{\delta}{w_{u_{n + 1}}}$ is $(27, h_{26})$-recurrent so we can apply Lemma \ref{lemma:contraction_for_Delta_recurrent_patterns} with $\mathcal{P} = \patternprime{w_{u_n}}{\delta}{w_{u_{n + 1}}}$, $\Delta \leq 27$, $N = u_{n - 1}$, $N' = u_{n}$ and we indeed have that $N' \ge N + 4 + 26\lceil\Delta/3\rceil$ since by hypothesis $N' \ge N + 302>N+238\ge N+4 + 26\lceil\Delta/3\rceil$.
        \item Finally, if $\delta = 3$ or $\delta \ge 9$, then by Lemma \ref{lemma:Delta_constructability_of_A3B_and_Ageq9B_bis} the pattern $\patternprime{w_{u_n}}{\delta}{w_{u_{n + 1}}}$ is $(10, h)$-constructible. Therefore, we can apply Lemma \ref{lemma:contraction_for_Delta_constructible_patterns} with $\mathcal{P} = \patternprime{w_{u_n}}{\delta}{w_{u_{n + 1}}}$, $N = u_{n - 1}$, $N' = u_n$, and $\Delta \leq 10$ and we indeed have that $N' \ge N + 26\lceil(\Delta+1)/3\rceil +198$ since by hypothesis $N' \ge N + 302\ge N+26\lceil(\Delta+1)/3\rceil + 198$.
    \end{itemize}
\end{itemize}
In each case, by either applying Lemma \ref{lemma:contraction_for_Delta_recurrent_patterns} or Lemma \ref{lemma:contraction_for_Delta_constructible_patterns}, we can modify $w'$ so that it is still  an infinite ternary square-free word, and for every integer $i \le n + 1$, $w'_{u_i} = w_{u_i}$.
\end{proof}

Notice that in Theorem \ref{thm:sufficient_condition}, there are no conditions on the subsequences modulo $p$ and $q$ besides condition (\ref{condAdjacent}), so this result could be used to prove existence of words with subsequences having other constraints than being square-free.

\subsection{Pairs that respect the conditions}\label{sec:pairs_that_respect_the_condition}

In this section, we prove the existence of infinite ternary words that are square-free modulo $p$, square-free modulo $q$, and that respect condition (\ref{condAdjacent}). By Theorem \ref{thm:sufficient_condition}, the pairs $(p, q)$ of relatively prime integers for which such words exist are positive pairs.

\begin{theorem}\label{thm:pq_coprime_st_w_respect_star_and_is_sqfree}
    Let $(p, q)$ be a pair of relatively prime integers such that $p$, $q \ge 3$ and $\max(p, q) \ge 364$. Then there exists an infinite ternary word $w$ that respects condition (\ref{condAdjacent}) and such that $\subsequence{w}{p}$ and $\subsequence{w}{q}$ are square-free.
\end{theorem}

To prove Theorem \ref{thm:pq_coprime_st_w_respect_star_and_is_sqfree}, we will fix $\subsequence{w}{q}$ to be an arbitrary square-free word (assuming $p<q$), and we will construct the square-free word $\subsequence{w}{p}$ while preserving the condition (\ref{condAdjacent}) and the fact that $\subsequence{w}{p}$ must agree with $\subsequence{w}{q}$ whenever they coincide (these positions are multiples of $q$ in $\subsequence{w}{p}$).
We first give some lemmas about the structure of the constraints for such words before proving Theorem \ref{thm:pq_coprime_st_w_respect_star_and_is_sqfree}.

\begin{lemma}\label{lemma:description_of_constraints_distance_for_p_q_coprime}
    Let $(p, q)$ be a pair of relatively prime integers with $p, q \ge 3$. There exist $a,b>0$ with $2a + b = q$ such that for every integer $i \ge 0$, the only multiples of $p$ in $]ipq, (i + 1)pq[$ that are congruent to $\pm 1$ modulo $q$ are $ipq + ap$, and $ipq + ap + bp$.
\end{lemma}

\begin{proof}
Since $p$ and $q$ are coprime, the Chinese remainder theorem implies that there exists $s_+\in ]0,pq[$ such that 
\[ 
\{x: x\equiv 0\mod p \text{ and } x\equiv +1\mod q \}=\{s_++ipq: i\in \mathbb{N}\}.
\]
That is, for all $i$, $s_++ipq$ is the unique multiple of $p$ in $]ipq, (i + 1)pq[$ that is congruent to $+1$ modulo $q$. Similarly, there exists $s_-\in ]0,pq[$ such that  for all $i$, $s_-+ipq$ is the unique multiple of $p$ in $]ipq, (i + 1)pq[$ that is congruent to $-1$ modulo $q$. 

Note that, $s_++s_-$ is congruent to $0$ both modulo $q$ and modulo $p$. Since $s_++s_-\in ]0,2pq[$, we deduce $s_++s_-=pq$.
Now take $a=\frac{\min(s_-, s_+)}{p}$ and $b= \frac{\max(s_-, s_+)}{p}-a$ and we have the desired property.
\end{proof}

In order to prove Theorem \ref{thm:pq_coprime_st_w_respect_star_and_is_sqfree}, we will need to find the patterns $\patternprime{\overline{a}}{\delta}{\overline{b}}$ or the patterns $\patternprime{\overline{a}}{\delta}{b}\patternprime{}{\delta}{\overline{c}}$ in a square-free word. In what follows, we provide results on their recurrence. We first define $\mathcal{P}_{bad}$, a set of bad patterns containing the following patterns,  that will be treated separately.\\
\begin{center}
\begin{tabular}{lll}
$\patternprime{\overline{a}}{16}{a}\patternprime{}{16}{\overline{a}}$ &$\patternprime{\overline{a}}{16}{a}\patternprime{}{16}{\overline{\pi_3^2(a)}}$
&$\patternprime{\overline{a}}{14}{a}\patternprime{}{14}{\overline{\pi_3(a)}}$ \\
$\patternprime{\overline{a}}{12}{a}\patternprime{}{12}{\overline{a}}$ 
&$\patternprime{\overline{a}}{12}{\pi_3^2(a)}\patternprime{}{12}{\overline{a}}$ 
&$\patternprime{\overline{a}}{12}{\pi_3(a)}\patternprime{}{12}{\overline{\pi_3^2(a)}}$ \\
$\patternprime{\overline{a}}{10}{\pi_3(a)}\patternprime{}{10}{\overline{a}}$ 
&$\patternprime{\overline{a}}{8}{\pi_3^2(a)}\patternprime{}{8}{\overline{a}}$
&$\patternprime{\overline{a}}{8}{\pi_3(a)}\patternprime{}{8}{\overline{\pi_3^2(a)}}$ \\
$\patternprime{\overline{a}}{6}{\pi_3^2(a)}\patternprime{}{6}{\overline{\pi_3(a)}}$ 
&$\patternprime{\overline{a}}{0}{\pi_3^2(a)}\patternprime{}{0}{\overline{a}}$ 
&$\patternprime{\overline{a}}{0}{\pi_3(a)}\patternprime{}{0}{\overline{a}}$ \\
\end{tabular} \\
\end{center}
and that for every  $a \in \Sigma_3$.

\begin{lemma}\label{lemma:pattern!a<-d->!b_distance_and_pattern!a<-d->b<-d->!c_distance} 
    For every $\delta \in \{0,\dots, 17\}$ and for every letter $a$, $b$, $c \in \Sigma_3$, 
    \begin{itemize}
        \item the pattern $\mathcal{P} = \patternprime{\overline{a}}{\delta}{\overline{b}}$ is $(6, h_{26})$-recurrent, and
        \item either the pattern $\mathcal{P} = \patternprime{\overline{a}}{\delta}{b}\patternprime{}{\delta}{\overline{c}}$ is in  $\mathcal{P}_{bad}$, or it is $(27, h_{26})$-recurrent.
    \end{itemize}
\end{lemma}

\begin{lemma}\label{lemma:Mbad_patterns_are_13_h_constructible}
    Every pattern $\mathcal{P} \in \mathcal{P}_{bad}$ is $(13, h)$-constructible.
\end{lemma}

Both Lemma \ref{lemma:pattern!a<-d->!b_distance_and_pattern!a<-d->b<-d->!c_distance} and Lemma \ref{lemma:Mbad_patterns_are_13_h_constructible}
can be checked by a computer program performing an exhaustive search on images of $h$. We prove the first item of Lemma \ref{lemma:pattern!a<-d->!b_distance_and_pattern!a<-d->b<-d->!c_distance} in \cite[lemma\_\ref{lemma:pattern!a<-d->!b_distance_and_pattern!a<-d->b<-d->!c_distance}\_i.cpp]{code} by generating all factors of length $30$ appearing in the images of $h_{26}$. Then, we check that every pattern $\mathcal{P} = \patternprime{\overline{a}}{\delta}{\overline{b}}$ appears at position at most $6$ in each factor. We proceed similarly in \cite[lemma\_\ref{lemma:pattern!a<-d->!b_distance_and_pattern!a<-d->b<-d->!c_distance}\_ii.cpp]{code} to prove the second item of Lemma \ref{lemma:pattern!a<-d->!b_distance_and_pattern!a<-d->b<-d->!c_distance} except that we generate the factors of length 70 instead and that we check that for every pattern $\mathcal{P} = \patternprime{\overline{a}}{\delta}{b}\patternprime{}{\delta}{\overline{c}}$ that is not in  $\mathcal{P}_{bad}$, $\mathcal{P}$ appears at position at most $27$ in each factor.
To prove Lemma \ref{lemma:Mbad_patterns_are_13_h_constructible}, for each pattern $\mathcal{P} \in \mathcal{P}_{bad}$ and each square-free pre-image $t$ of length $2$, we compute by exhaustive search a guiding sequence $\Gamma_{\mathcal{P}, t}$ minimizing the position of the first occurrence of $\mathcal{P}$ in $h_{\Gamma_{\mathcal{P}, t}}(t)$, and we check that this position is at most $13$; this verification is performed in \cite[lemma\_\ref{lemma:Mbad_patterns_are_13_h_constructible}.cpp]{code}.

Another way to look at patterns is to consider them as partial words.
A \emph{partial word} over an alphabet $\Sigma$ is a (possibly infinite) word over $\Sigma \cup \{\diamond\}$. For any (finite or infinite) partial word $v$ and (finite or infinite) word $w$, we say that $w$ is \emph{compatible} with $v$ if $|w| \le |v|$ (or both are infinite) and for every $i \le |w|$, if $v_i \ne \diamond $ then $w_i = v_i$. For instance, the word $chocolate$ is compatible with the word $c\diamond o\diamond\diamond\diamond at\diamond$ and is not compatible with the word $c\diamond\diamond a\diamond c\diamond ao\diamond$. If a partial word has infinite length, then it can be seen as a set of forced letters at some positions. The following theorem states that if such constraints are at distance at least $19$ from each other, then we can construct a square-free word fulfilling those constraints.

\begin{theorem}[Rosenfeld, 2020 \cite{rosenfeld2020far}]\label{thm:borne_19_square_free_from_forced_letters}
    Let $(p_i)_{i \ge 0} \in \mathbb{N}_{\ge 18}^\omega$
    and $(c_i)_{i \ge 0} \in \Sigma^\omega$ be two sequences. 
    For any partial word $v = \prod\limits_{i\ge 0}c_i\diamond^{p_i}$, there exists an infinite ternary square-free word $w$ compatible with $v$.  
\end{theorem}

We are now ready to prove Theorem \ref{thm:pq_coprime_st_w_respect_star_and_is_sqfree} :

\begin{proof}[Proof of Theorem \ref{thm:pq_coprime_st_w_respect_star_and_is_sqfree}]
    Without loss of generality, assume that $p \le q$.
    Let us fix $\subsequence{w}{q}$ to be any infinite ternary square-free word $s$. Since $p$ and $q$ are coprime, and $\subsequence{w}{q}$ is fixed, there are some constraints on $\subsequence{w}{p}$ where either a letter is forced (when $\subsequence{w}{p}$ and $\subsequence{w}{q}$ coincide in $w$), or a letter is forbidden (when $\subsequence{w}{p}$ and $\subsequence{w}{q}$ are at distance one in $w$). 
    Let $(u_i)_{i \ge 0}$ be the ordered sequence of the positions of the constraints on $\subsequence{w}{p}$ in $\subsequence{w}{q}$ and let $(v_i)_{i \ge 0}$ be the ordered sequence of the positions of the constraints on $\subsequence{w}{p}$ in $\subsequence{w}{p}$. That is, if the $i$-th constraint on $\subsequence{w}{p}$ is that ${\subsequence{w}{p}}_j$ must be equal to or different from ${\subsequence{w}{q}}_{j'}$, then $u_i = j'$ and $v_i = j$.
    
    By Lemma \ref{lemma:description_of_constraints_distance_for_p_q_coprime}, there exist exactly two integers $a$, $b \ge 0$ such that $2a + b = q$ and for every integer $i \ge 0$, $v_{3i} = iq$, $v_{3i + 1} = iq + a$, $v_{3i + 2} = iq + a + b$, and $v_{3i + 3} = iq + 2a + b = (i + 1)q$. 
    In particular, $\subsequence{w}{p}$ must be compatible with the partial word 
    $$w' = \prod_{i\ge0}s_{u_{3i}}\diamond^{a - 1}\overline{s_{u_{3i + 1}}}\diamond^{b - 1}\overline{s_{u_{3i + 2}}}\diamond^{a - 1}.$$ 
    
    If $a \ge 19$ and $b \ge 19$, then by Theorem \ref{thm:borne_19_square_free_from_forced_letters}, there exists an infinite ternary square-free word $w''$ compatible with $w'$, and we can simply set $\subsequence{w}{p} = w''$. Otherwise, let $\Gamma = (\gamma_i)_{i \ge 0}$ be the guiding sequence of constant value $26$, and let $t$ be an infinite ternary square-free word such that $h_\Gamma(t)_0 = s_0$. By construction, there exists $n \ge 0$ such that the prefix of length $v_n + 1$ of $h_\Gamma(t)$ is compatible with $w'$. We will show how to increase the value of $n$ by only doing modifications on $\Gamma$ so that $h_\Gamma(t)$ remains square-free.
\begin{itemize}
    \item If $a \le 18$, then by definition of $a$ and $b$, $b \ge q - 2a \ge 328$ since by hypothesis $q \ge 364$.  Let $m \le 2$ be such that $v_{n - m} - v_{n - m - 1} = b$, and let $\mathcal{P} = \pattern{\overline{s_{u_{n - m}}}}{a}{s_{u_{n - m + 1}}}\pattern{}{a}{\overline{s_{u_{n - m + 2}}}}$. 
    \begin{itemize}
        \item If $\mathcal{P} \in \mathcal{P}_{bad}$, then $\mathcal{P}$ is $(13, h)$-constructible by Lemma \ref{lemma:Mbad_patterns_are_13_h_constructible}. Therefore, we can apply Lemma \ref{lemma:contraction_for_Delta_constructible_patterns} with $N = v_{n - m - 1}$, $N' = v_{n - m}$, $\mathcal{P}$, and $\Delta \le 13$ and we indeed have that $N' \ge N + 26\lceil(\Delta + 1)/3\rceil + 198$ since by hypothesis, $N' - N = b \ge 328$ and $26\lceil(\Delta + 1)/3\rceil + 198 \le 328$.
        \item If $\mathcal{P} \not\in \mathcal{P}_{bad}$, then $\mathcal{P}$ is $(27, h_{26})$-recurrent by Lemma \ref{lemma:pattern!a<-d->!b_distance_and_pattern!a<-d->b<-d->!c_distance} so we can apply 
        Lemma \ref{lemma:contraction_for_Delta_recurrent_patterns} with $N = v_{n - m - 1}$, $N' = v_{n - m}$, $\mathcal{P}$, and $\Delta \le 27$ and we indeed have that 
        $N'\ge N + 4 + 26\lceil\Delta/3\rceil$ 
        since by hypothesis, $N' - N = b \ge 328$ and $4 + 26\lceil\Delta/3\rceil \le 238$.
    \end{itemize}
    \item If $b \le 18$, then by definition of $a$ and $b$, $a \ge (q - b)/2 \ge 173$ since by hypothesis, $q \ge 364$.
    \begin{itemize}
        \item If $v_{n + 1} - v_n = a$, then we can apply Lemma \ref{lemma:contraction_for_Delta_recurrent_patterns} with $N = v_n$, $N' = v_{n + 1}$, $\mathcal{P} = s_{u_{n + 1}}$, and $\Delta = 3$ since the pattern $a$ is $(3, h_{26})$-recurrent, and we indeed have that $N' \ge N + 4 + 26\lceil\Delta/3\rceil$ since by hypothesis,  $N' - N = a \ge 173$ and $4 + 26\lceil\Delta/3\rceil \le 30$.
        \item Otherwise, $v_{n + 1} - v_n = b$. By Lemma \ref{lemma:pattern!a<-d->!b_distance_and_pattern!a<-d->b<-d->!c_distance}, the pattern $\mathcal{P} = \pattern{\overline{s_{u_n}}}{b}{\overline{s_{u_{n + 1}}}}$ is $(6, h_{26})$-recurrent, we can thus apply Lemma \ref{lemma:contraction_for_Delta_recurrent_patterns} with $N = v_{n - 1}$, $N' = v_n$, $\mathcal{P} = \pattern{\overline{s_{u_n}}}{b}{\overline{s_{u_{n + 1}}}}$, and $\Delta \le 6$ and we indeed have that $N' \ge N + 4 + 26\lceil\Delta/3\rceil$ since by hypothesis,  $N' - N = a \ge 173$ and $4 + 26\lceil\Delta/3\rceil \le 56$.
    \end{itemize}
\end{itemize}

In all cases, by applying Lemma \ref{lemma:contraction_for_Delta_recurrent_patterns} or Lemma \ref{lemma:contraction_for_Delta_constructible_patterns}, we can construct a guiding sequence $\Gamma'$ such that the prefix of length $v_{n + 1} + 1$ of $h_{\Gamma'}(t)$ is compatible with $w'$, as desired.
\end{proof}

We can finally apply Theorem \ref{thm:sufficient_condition} together with Theorem \ref{thm:pq_coprime_st_w_respect_star_and_is_sqfree} to deduce the main result of this Section.
\begin{theorem}\label{thm:pq_large}
    For all relatively prime integers $p, q \ge 331$ such that $\max(p, q) \ge 364$, there exists an infinite ternary square-free word that is square-free modulo $p$ and $q$.
\end{theorem}

\section{Large $q$ and small $p$}\label{chap:smallplargeq}

In Section \ref{chap:pqlarge}, we proved that when $p$ and $q$ are coprime and large enough, there always exists an infinite ternary square-free word that is square-free modulo $p$ and square-free modulo $q$. In this section, we will present another strategy that will allow us to deal with pairs with $p$ small and $q$ large (and symmetrically, with $p$ large and $q$ small). Our strategy is based on the existence of some morphisms and on the notion of compatible words.

The \emph{shift} of a word $w = w_0w_1w_2\dots$, denoted by $s(w)$, is the word $w_1w_2\dots$. If we consider a word $w$, then for every integer $\alpha \ge 0$ and $p \ge 1$, $\subsequence{s^\alpha(w)}{p} = w_\alpha w_{\alpha + p}w_{\alpha + 2p}\dots$ and we call this word the \emph{subsequence congruent to $\alpha$ modulo $p$ of $w$}.
The following proposition provides a sufficient condition on $p$ and $q$ for the existence of an infinite ternary square-free word that is also square-free modulo $p$ and square-free modulo $q$.

\begin{proposition}\label{prop:circular_morphism_sf_sfmodp}
    Let $p \ge 1$. If there exists a morphism $g$ such that 
    \begin{itemize}
        \item $g$ is circular,
        \item $g$ is $kp$-uniform for some $k \ge 1$,  
        \item $g$ is square-free, and
        \item there exists $\alpha \ge 0$ such that for every infinite ternary square-free word $t$, $\subsequence{s^\alpha(g(t))}{p}$ is square-free,
    \end{itemize}
    then for every integer $q \ge 19kp$, 
    there exists an infinite ternary square-free word $w$ such that $\subsequence{w}{p}$ and $\subsequence{w}{q}$ are square-free.
\end{proposition}

\begin{proof}
Let $g$ be a morphism as in the statement of Proposition \ref{prop:circular_morphism_sf_sfmodp}.
By hypothesis, for every infinite ternary square-free word $w$, $s^\alpha(g(w))$ and $\subsequence{s^\alpha(g(w))}{p}$ are square-free.
Let us fix $t$ to be any infinite ternary square-free word. We would like to find a word $w$ such that $\subsequence{s^\alpha(g(w))}{q} = t$ and doing so, $s^\alpha(g(w))$, $\subsequence{s^\alpha(g(w))}{p}$ and $\subsequence{s^\alpha(g(w))}{q}$ would be square-free. 

Since $g$ is circular, for every integer $i < kp$, $\{g(0)_i, g(1)_i, g(2)_i\}= \Sigma_3$. Let $i \ge 0$ and let $d$, $r$ be the quotient and the remainder in the Euclidean division of $iq$ by $kp$ so that $iq = dkp + r$. Since $t$ is fixed, $g(w)_{dkp + r}$ must be equal to $t_i$. Since $g$ is circular, and by the previous observation, this constraint is actually a constraint on $w_d$. Let $d'$, $r'$ be the quotient and the remainder in the Euclidean division of $(i + 1)q$ by $kp$. We have that $(i + 1)q - iq = q = (d' - d)kp + r' - r \ge 19kp$ so $d' - d \geq \frac{19kp - r + r'}{kp} > 18$ and since $d' - d$ is an integer, then $d' - d \geq 19$. Let $(u_i)_{i \ge 0}$ be the sequence of positions of $w$ that have forced values. The previous observation implies that for every integer $i$, $u_{i + 1} - u_i \ge 19$. Therefore, by Theorem \ref{thm:borne_19_square_free_from_forced_letters}, $w$ exists.
\end{proof}

In order for us to apply Proposition \ref{prop:circular_morphism_sf_sfmodp}, we need such morphisms. It is known that such morphisms exist for almost all integers $p$ with $k = 1$ and $\alpha = 0$. Notice that if a morphism $g$ is circular with images of size $p$, then in particular it is square-free modulo $p$ (in this case, up to a permutation of the alphabet, $\subsequence{g(w)}{p}=w$).
\begin{proposition}[Currie, 2012 \cite{currie2012infinite}]\label{prop:currie2021infinite}
    For every integer $p \in \{13, 17, 18, 19\} \cup \mathbb{N}_{\ge 23}$, there exists a square-free circular morphism with images of length $p$.
\end{proposition}

We can therefore apply Proposition \ref{prop:circular_morphism_sf_sfmodp} with $k = 1$ and $\alpha = 0$ and with the morphisms from Proposition \ref{prop:currie2021infinite} to obtain the following family of pairs of integers $p$, $q$ for which there exists an infinite ternary square-free word that is square-free modulo $p$ and $q$.

\begin{corollary}\label{coro:construction_from_cyclic_morphisms} 
    For all integers $p$, $q$ such that $p \in \{13, 17, 18, 19\} \cup \mathbb{N}_{\ge 23}$ and $q \ge 19 p$, there exists an infinite ternary square-free word $w$ such that $\subsequence{w}{p}$ and $\subsequence{w}{q}$ are square-free.
\end{corollary}

We provide in Table \ref{fig:morphisms_cone_remaining_cases} (also in \cite[table\_\ref{fig:morphisms_cone_remaining_cases}\_morphisms.txt]{code}) morphisms for all remaining cases other than $p=6$. 
More precisely, Table \ref{fig:morphisms_cone_remaining_cases} contains for all $p$ not covered by Corollary \ref{coro:construction_from_cyclic_morphisms} and different from $6$ the morphism, the parameters $k$ and $\alpha$ used in Proposition \ref{prop:circular_morphism_sf_sfmodp}, and the upper bound on the smallest value of $q$ from which there exist infinite ternary square-free words that are square-free modulo $p$ and square-free modulo $q$. This leads to the following Proposition.
\begin{proposition}\label{prop:sparse_p_pairs}
    For every $p \in \{3, 4, 5, 7, 8, 9, 10, 11, 12, 14, 15, 16, 20, 21, 22\}$, there exists $l$ such that for every $q \geq l$, there exists an infinite ternary square-free word that is square-free modulo $p$ and square-free modulo $q$.
\end{proposition}

An implementation of an algorithm verifying that the morphisms from Table \ref{fig:morphisms_cone_remaining_cases} have the desired properties can be found in \cite[table\_\ref{fig:morphisms_cone_remaining_cases}.cpp]{code}. 
The verification uses the following characterization of square-free morphisms due to Crochemore.

\begin{theorem}[Crochemore, 1982~\cite{CROCHEMORE1982221}]\label{thm:crochemore}
    Let $g : \Sigma^* \rightarrow \Sigma^*$ be a morphism.
    \begin{enumerate}
        \item If $|\Sigma| = 3$ then $g$ is square-free if and only if for every square-free word $w$ over $\Sigma$ of length $5$, $g(w)$ is square-free.
        \item If $g$ is uniform then $g$ is square-free if and only if for every square-free word $w$ over $\Sigma$ of length $3$, $g(w)$ is square-free.
    \end{enumerate}
\end{theorem}

If $g : \Sigma_3^* \rightarrow \Sigma_3^*$ is a $kp$-uniform morphism for some integers $k$, $p \geq 1$, then let $g^{\alpha, p} : \Sigma_3^* \rightarrow \Sigma_3^*$ be the morphism such that $g^{\alpha, p}(a) = \subsequence{s^\alpha(g(a))}{p}$ for every $a \in \Sigma_3$. Observe that for every $\alpha \in \{0, \dots, p - 1\}$, and for every ternary word $t$, $\subsequence{s^\alpha(g(t))}{p} = g^{\alpha, p}(t)$, so to verify that $g$ satisfies all the hypotheses of Proposition \ref{prop:circular_morphism_sf_sfmodp}, it is sufficient to use Theorem \ref{thm:crochemore} on both $g$ and $g^{\alpha, p}$.

\section{Large $q$ and $p=6$}\label{sec:p_equal_6}
In the previous section, we proved that for every $p\not\in\{2,6\}$ and every large enough $q$, there exists an infinite ternary square-free word that is square-free modulo $p$ and square-free modulo $q$. When $p=2$, it is not possible to construct an infinite ternary square-free word that is also square-free modulo $2$.
This section is thus dedicated to the remaining case $p = 6$. 

Let $R_{01}$, $R_{02}$, $R_{10}$, $R_{12}$, $R_{20}$, $R_{21}$, be ternary words defined as follows :

\begin{center}
\begin{tabular}{ c c}
 $R_{01} \coloneqq\textbf{0}12102$ & $R_{02} \coloneqq\textbf{0}12021$ \\ 
 $R_{10} \coloneqq\textbf{1}20102$ & $R_{12} \coloneqq\textbf{1}20210$ \\
 $R_{20} \coloneqq\textbf{2}01021$ & $R_{21} \coloneqq\textbf{2}01210$  
\end{tabular}
\end{center}

Observe that for every $i$, $j$, $R_{ij} = iabcde$ where $a = \pi_3(i)$, $b = \pi_3(a)$, $c = \pi_3(d)$, and $d = \pi_3(e)$ and this observation, together with the fact that $R_{ij}j$ must be square-free, uniquely determines each of the $R_{ij}$.
For every ternary square-free word $t$, we define the \emph{completion} of $t$ by $\mathcal{R}$ as 
$$\mathcal{R}(t) = \prod_{i \ge 0}^{|t| - 1}R_{t_it_{i + 1}}$$
and this definition naturally extends to infinite words.
For instance, $\mathcal{R}(012) = \textbf{0}12102\textbf{1}20210$.
The rule set $\mathcal{R}$ can be seen as a functional finite-state transducer as shown in Figure \ref{fig:6-completion_transducer}. This notion of rule set is also closely related to the notion of two-blocks substitutions as studied in \cite{DEKKING2023102536}. 
\begin{figure}[H]
    \centering
    \includegraphics[scale=0.5]{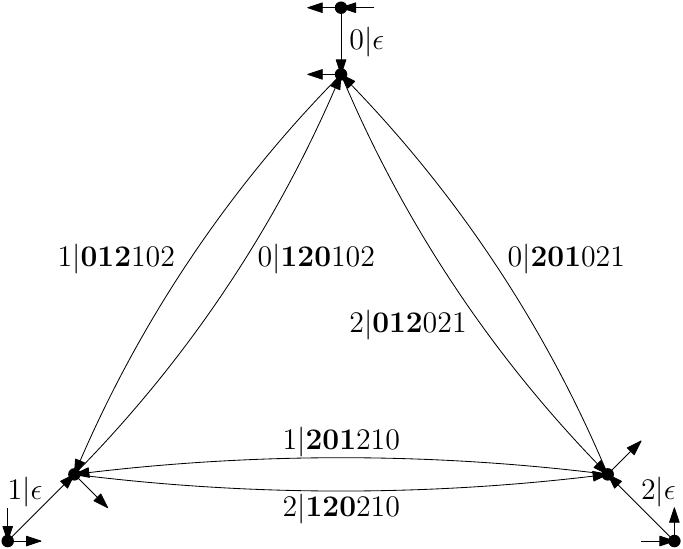}
    \caption{$\mathcal{R}$ represented as a finite-state transducer.}
    \label{fig:6-completion_transducer}
\end{figure}

Notice that $\subsequence{\mathcal{R}(t)}{6} = t'$ where $t'$ is the prefix of size $|t| - 1$ of $t$. Hence, if $t$ is square-free, then $\mathcal{R}(t)$ is square-free modulo $6$. Proposition \ref{prop:6-compl-sqfree} also states that applying $\mathcal{R}$ to square-free words only produces square-free words, thus, with $\mathcal{R}$, we can construct words that are square-free and square-free modulo $6$.
\begin{proposition}\label{prop:6-compl-sqfree}
    For every (infinite) ternary square-free word $t$, $\mathcal{R}(t)$ is an (infinite) ternary square-free word.
\end{proposition}

\begin{proof}
    Assume for the sake of contradiction that $\mathcal{R}$ does not preserve square-freeness. So let $t$ be a finite square-free word such that $\mathcal{R}(t)$ contains a square $uu$ of period $\Delta$ for some $\Delta \geq 1$. A simple inspection of $\mathcal{R}$ shows that $t$ must be of length at least $2$ so there exists a non-empty ternary square-free word $t' \in \Sigma_3^+$ and $x \in \Sigma_3$ such that $t = t'x$.
    Assume that $\Delta \ge 10$. Hence, the first occurrence of $u$ contains at some position $d\in\{0,1,2,3,4,5\}$ the beginning of a block $R_{ij}$ for some $i$, $j \in \{0, 1, 2\}$ which is a factor $abc$ with $b = \pi_3(a)$ and $c = \pi_3(b)$.

    Since $\pi_3\circ\mathcal{R} = \mathcal{R}\circ\pi_3$, we can assume without loss of generality that $abc = 012$. As already noted, $\subsequence{\mathcal{R}(t)}{6} = t'$, so since $t$ is square-free, then so is $\subsequence{\mathcal{R}(t)}{6}$. If $\Delta \equiv 0 \pmod6$, then for all $i \in \{0, 1, 2, 3, 4, 5\}$ the subsequence congruent to $i$ modulo $6$ of $uu$ is a square and for one of these $i$ this sequence is a factor of $\subsequence{\mathcal{R}(t)}{6}$ which is absurd since $\subsequence{\mathcal{R}(t)}{6}$ is square-free.
    
    Therefore, we can assume that $\Delta \equiv \alpha\pmod6$ with $0 < \alpha < 6$. 
    Observe that the factor $012$ only appears at position $0$ of $R_{01}$, at position $0$ of $R_{02}$, at position $1$ of $R_{21}$, at position $5$ of $R_{21}R_{10}$, or at position $5$ of $R_{21}R_{12}$ (remember that by definition of $\mathcal{R}$, $R_{12}R_{10}$ or $R_{12}R_{12}$ cannot appear in the image of a word by $\mathcal{R}$). By hypothesis, the factor $012$ at position $d$ of the first period is created by either the prefix of size $3$ of $R_{01}$ or of $R_{02}$. Since $0 < \alpha < 6$, the factor $012$ at position $d$ of the second period must be created by either $R_{21}$, $R_{21}R_{10}$, or $R_{21}R_{12}$ . 
    \begin{itemize}
        \item If $d \geq 4$, then the factor $012$ at position $d$ of the first period must be preceded by $R_{10}$ or $R_{20}$, so $u$ contains either the factor $102012$ or the factor $1021012$ at position $d - 4$. If the factor $012$ at position $d$ of the second period is created by $R_{21}$, then similarly, it must be preceded by $R_{02}2$ or $R_{12}2$, so $u$ must contain either the factor $0212012$ or the factor $2102012$ at position $d - 4$ which is a contradiction. Otherwise, if the factor $012$ at position $d$ of the second period is created by $R_{21}R_{10}$ or $R_{21}R_{12}$, then $u$ must contain either the factor $0121012$ or the factor $2021012$ at position $d - 4$ which is a contradiction.
        \item If $d < 4$, then the factor $012$ at position $d$ of the first period must be followed by either $102$ or $021$ depending on whether $012$ is created by $R_{01}$ or $R_{02}$, so since $\Delta \geq 10$ and $d < 4$, $u$ must contain either the factor $0121021$ or the factor $0120212$ at position $d$. If the factor $012$ at position $d$ of the second period is created by $R_{21}$, then it is followed by either $10R_{10}$ or $10R_{12}$. In both cases, $u$ must contain the factor $0121012$ at position $d$ which is a contradiction. Otherwise, if the factor $012$ at position $d$ of the second period is created by $R_{21}R_{10}$ or $R_{21}R_{12}$, then $u$ must contain either the factor $0120102$ or the factor $0120210$ at position $d$ which is a contradiction.
    \end{itemize}
    Finally, if $\Delta \leq 9$, the square $uu$ in $\mathcal{R}(t)$ must start at position $d \leq 5$ in the image of a square-free word of length $5$ under $\mathcal{R}$. To verify that no such word exists\footnote{It can also easily be done with a simple computer assisted verification.}, it is sufficient to look at the image of a square-free word of length $4$ starting with $01$ or $02$ under $\mathcal{R}$ and observe, for each of the following words, for each period $\Delta \in \{1, 2, 3, 4, 5, 6, 7, 8, 9\}$ and for each position $d \in \{0, 1, 2, 3, 4, 5\}$, that the factor of length $\min(\Delta, 5)$ starting at position $d$ is always different from the factor of the same length starting at position $d + \Delta$:
    \begin{align*}
        &\textbf{0}12102\textbf{1}20102\textbf{0}12021\textbf{2}, \textbf{0}12102\textbf{1}20210\textbf{2}01021\textbf{0},
    \textbf{0}12102\textbf{1}20210\textbf{2}01210\textbf{1},\\
        &\textbf{0}12021\textbf{2}01021\textbf{0}12102\textbf{1},
    \textbf{0}12021\textbf{2}01210\textbf{1}20102\textbf{0},
    \textbf{0}12021\textbf{2}01210\textbf{1}20210\textbf{2}. \qedhere
    \end{align*}
\end{proof}

In order to construct words that are square-free, square-free modulo $6$, and square-free modulo $q$, we would like to shift some patterns to fulfill constraints as previously. However, contractions on a word $t$ shift patterns in $\mathcal{R}(t)$ by distances that are multiple of $6$. Therefore, we need some recurrence result on the subsequences that are congruent to $\alpha$ modulo $6$ for every $\alpha \in \{0, 1, 2, 3, 4, 5\}$.

\begin{lemma}\label{lemma:factors_of_sf_words_of_length_8}
    For every ternary square-free word $w$ of length $8$, $w$ has at least $5$ different factors of length $2$.
\end{lemma}

\begin{proof}
    We can explore all the possible values for $w$, assuming $w_0w_1 = 01$. This leads to the following tree of possibilities. The tree is oriented with all arcs pointing away from the root.
    \begin{center}
    \includegraphics[scale = 0.65]{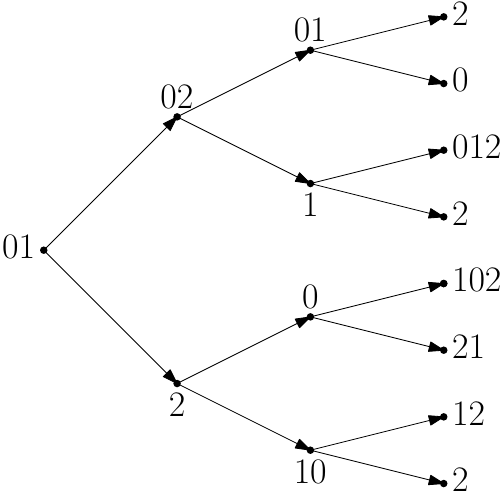}
    \end{center}
    A root-to-leaf oriented path labeled with the letters $a$, $b$, $c$, $d$ represents the word $abcd$. Observe that every root-to-leaf oriented path contains at least $5$ different factors of length $2$ except for the path $0102010$ but this word cannot be a factor of a square-free word of length $8$. Moreover, every square-free word of length $8$ starting with $01$ must have a prefix that is a root-to-leaf oriented path in the tree of possibilities since the tree explores all possibilities to construct a square-free word starting with $01$. Finally, the longest root-to-leaf oriented path is $01021012$ of length $8$.
\end{proof}

\begin{corollary}\label{cor:u_v_size_2_6_recurrent}
    For every pair of ternary square-free words $u$ and $v$ of length $2$ with $u \neq v$, the pattern $\{u, v\}$ is $(6, h_{26})$-recurrent.
\end{corollary}

\begin{lemma}\label{lemma:shift_in_6_completions}
Let $t$ be an infinite ternary square-free word, and let $\Gamma = (\gamma_i)_{i \ge 0}$ be a guiding sequence. For every integer $M \ge 0$, $M' \ge M + 341$ and for every letter $a \in \Sigma$, 
there exists a guiding sequence $\Gamma'$ such that
    \begin{itemize}
        \item for every integer $i \le M$, $\mathcal{R}(h_{\Gamma'}(t))_i = \mathcal{R}(h_\Gamma(t))_i$, and
        \item $\mathcal{R}(h_{\Gamma'}(t))_{M'} = a$. 
    \end{itemize}
\end{lemma}
\begin{proof}
    By inspection of the $R_{ij}$, for every $\alpha \in [0, 5]$, there exist exactly two words $u$, $v$ with 
    \[u,v \in \{01, 02, 10, 12, 20, 21\}\]
    such that $(R_u)_\alpha = a$ and $(R_v)_\alpha = a$. Therefore for every infinite ternary square-free word $w$, $\mathcal{R}(w)_{6M'' + \alpha} = a$ if and only if $w$ contains the pattern $\mathcal{P} = \{u, v\}$ at position $M''$.
    We can thus reformulate our Lemma to ask for a guiding sequence $\Gamma'$ such that for every integer $i \le M$, $\mathcal{R}(h_\Gamma(t))_i = \mathcal{R}(h_{\Gamma'}(t))_i$ and for the pattern $\mathcal{P}$ to occur at position $M''$ in $h_{\Gamma'}(t)$ where 
    $M' = 6M'' + \alpha$ with $\alpha \in [0, 5]$. By Corollary  \ref{cor:u_v_size_2_6_recurrent}, $\mathcal{P}$ is a $(6, h_{26})$-recurrent pattern therefore we can apply Lemma \ref{lemma:contraction_for_Delta_recurrent_patterns} on $h_\Gamma(t)$ with $N = \lceil M/6\rceil$, $N' = M''$, $\mathcal{P} = \{u, v\}$ and $\Delta = 6$ to construct $\Gamma'$ and we indeed have that $N' \ge N + 4 + 26\lceil\Delta/3\rceil$ since by hypothesis, $M' \ge M + 341$, so $6M''+\alpha \ge  M + 341$ and therefore $M'' + \alpha/6 \ge (M + 341)/6 \ge \lceil M/6\rceil + 56$ so since $M$ and $M''$ are integers, $M'' \ge \lceil M/6\rceil + 56$
    and $4 + 26\lceil\Delta/3\rceil = 56$ as desired.
\end{proof}

\begin{proposition}\label{prop:p=6}
    For every integer $q \ge 341$, there exists an infinite ternary square-free word that is square-free modulo $6$ and square-free modulo $q$.
\end{proposition}

\begin{proof}
Let $s$ and $t$ be any infinite ternary square-free words such that $s_0 = t_0$ and let $\Gamma = (\gamma_i)_{i \ge 0}$ be the guiding sequence of constant value $26$. We will modify $\Gamma$ so that $\subsequence{\mathcal{R}(h_\Gamma(t))}{q} = s$. Notice that since $t$ is square-free, by Proposition \ref{prop:6-compl-sqfree}, $\mathcal{R}(h_\Gamma(t))$ is square-free and square-free modulo $6$. By definition of $s$ and $t$, there exists $n \ge 0$ such that $n$ is a multiple of $q$ and the prefix of size $n/q$ of $\subsequence{\mathcal{R}(h_\Gamma(t))}{q}$ is the prefix of size $n/q$ of $s$. To increase the value of $n$, we need to modify $\Gamma$ so that the prefix of length $n$ of $\mathcal{R}(h_\Gamma(t))$ remains the same and $\mathcal{R}(h_\Gamma(t))_{n + q} = s_{n/q + 1}$. Since by hypothesis $q \ge 341$, we can apply Lemma \ref{lemma:shift_in_6_completions} with $M = n$ and $M' = n + q$.
\end{proof}

We can then conclude this Section by combining Corollary \ref{coro:construction_from_cyclic_morphisms}, Proposition \ref{prop:sparse_p_pairs} and Proposition \ref{prop:p=6} to get the following Theorem.

\begin{theorem}\label{thm:psmall_qlarge}
    For every integer $p \ge 3$, there exists $l$ such that for every $q \geq l$, there exists an infinite ternary square-free word that is square-free modulo $p$ and square-free modulo $q$.
\end{theorem}

Notice that for every $p \geq 3$, Corollary \ref{coro:construction_from_cyclic_morphisms}, Proposition \ref{prop:sparse_p_pairs} or Proposition \ref{prop:p=6} provide an explicit value for $l$.

\section{Some of the remaining cases}\label{sec:remaining_cases}

Theorem \ref{thm:pq_large} and Theorem \ref{thm:psmall_qlarge} leave only finitely many pairs of relatively prime integers $(p, q)$ with $p$, $q \ge 3$ for which we do not know whether there exist infinite ternary square-free words that are square-free modulo $p$ and $q$. Using a computer, we verified that the number of pairs of relatively prime integers that remain unresolved is exactly $1238408$. To start addressing these remaining cases, we solve the problem for all pairs of integers $p$, $q \leq 20$ other than the pair $(5, 8)$ (and its symmetric pair $(8, 5)$).

\begin{proposition}
    Let $p$, $q \le 20$ be such that $p \le q$ and $(p, q) \neq (5,8)$ and let $S$ be the set 
    \begin{align*}
        &\{(3, 4), (3, 5), (3, 7), (3, 8), (3,9), (3,10), (4,5), (4,7), (4,9), (4,10), (4,14), (6,7)\} \\
        &\cup \{(2, i), i \ge 1\} \cup \{(i, 2i) \mid i \ge 1\} \cup \{(2i, 3i) \mid i \ge 1\}.
    \end{align*}
    There exists an infinite ternary square-free word that is square-free modulo $p$ and square-free modulo $q$ if and only if $(p, q) \not\in S$.
\end{proposition}
For the positive pairs, we found explicit morphisms that construct, from any ternary square-free word, a longer ternary square-free word that is square-free modulo $p$ and $q$. These morphisms can be found in \cite[table\_\ref{table:pq_leq_20}\_morphisms.txt]{code}, and the algorithm that verifies that these morphisms are square-free, square-free modulo $p$, and square-free modulo $q$ can be found in \cite[table\_\ref{table:pq_leq_20}\_blue.cpp]{code}.

For a ternary square-free word $w$ of length at least $3$, the \emph{Pansiot code} of $w$ is a binary word $c$ with $c_i = 1$ if $w_i = w_{i + 2}$ and $c_i = 0$ otherwise. To find a morphism for a given positive pair $(p,q)$, we first look for a suitable
set of Pansiot codes of length $\ell=\lcm(p,q)$ that could start at positions~$0\pmod{\ell}$ in a large ternary square-free word square-free modulo $p$ and square-free modulo $q$.
We then iterate the following steps:
\begin{itemize}
    \item We construct a large ternary square-free word $w$ square-free modulo $p$ and modulo $q$ such that the Pansiot code of each factor of length $\ell$ and at a position congruent to $0$ modulo $\ell$ of $w$ is in our set of suitable Pansiot codes.
    \item Then we forbid the Pansiot code of length $\ell$ that appears the least at a position congruent to $0$ modulo $\ell$ in $w$.
\end{itemize}
We continue until the set of surviving Pansiot codes gets reasonably small, but remains large enough to construct a
long ternary square-free word that is square-free modulo $p$ and modulo $q$.
We then try to find a morphism that uses only this set of Pansiot codes.

We proved that some pairs are negative using an exhaustive search algorithm that looks for the longest possible ternary square-free word that is also square-free modulo $p$ and $q$ in lexicographic order. If the search terminates, then no such infinite word exists. An implementation of this backtracking algorithm can be found in
\cite[table\_\ref{table:pq_leq_20}\_red.cpp]{code}.

In addition, we can eliminate all the pairs where $p$ or $q$ is equal to $2$. We can also eliminate many pairs of non-coprime integers with the following observations.
\begin{observation}\label{prop:non-square-free-words}
    For every integer $k$, $p$, $q \geq 1$, if there exists an infinite ternary square-free word $w'$ such that $\subsequence{w'}{kp}$ and $\subsequence{w'}{kq}$ are square-free, then there exists an infinite ternary word $w$ such that $\subsequence{w}{p}$ and $\subsequence{w}{q}$ are square-free.
\end{observation}

Indeed, consider the word $w$ such that $w_i = w'_{ki}$. Then $\subsequence{w}{p} = \subsequence{w'}{kp}$ and $\subsequence{w}{q} = \subsequence{w'}{kq}$, so $w$ is square-free modulo $p$ and $q$. Since there is no infinite ternary word that is square-free and square-free modulo $2$, and no infinite ternary word whose subsequences modulo $2$ and $3$ are both square-free, it follows that for every integer $k \geq 1$, there is no infinite ternary square-free word that is square-free modulo $k$ and $2k$, nor any such word that is square-free modulo $2k$ and $3k$.

For $p$, $q \leq 20$, both the negative pairs obtained from the backtracking algorithm and the positive pairs obtained with explicit morphisms are presented in Table \ref{table:pq_leq_20}. We conjecture that for every pair of relatively prime integers $p$ and $q$ such that $p$, $q \geq 3$ and $p + q \geq 14$, there exists an infinite ternary square-free word that is square-free modulo $p$ and $q$.

\newcolumntype{C}[1]{>{\arraybackslash}m{#1}}
\definecolor{myblue}{RGB}{48,154,235}
\definecolor{myred}{RGB}{245,0,0}
\definecolor{myblack}{RGB}{0,0,0}

\begin{table}[H]
\begin{center}
\scalebox{0.8}{
\begin{tabular}{|c|C{0.25cm}|C{0.25cm}|C{0.25cm}|C{0.25cm}|C{0.25cm}|C{0.25cm}|C{0.25cm}|C{0.25cm}|C{0.25cm}|C{0.25cm}|C{0.25cm}|C{0.25cm}|C{0.25cm}|C{0.25cm}|C{0.25cm}|C{0.25cm}|C{0.25cm}|C{0.25cm}|C{0.25cm}|C{0.25cm}|C{0.25cm}|C{0.25cm}|C{0.25cm}|}
\hline
\rule{0pt}{0.5cm}20 & \cellcolor{myblue} & \cellcolor{myred} & \cellcolor{myblue} & \cellcolor{myblue} & \cellcolor{myblue} & \cellcolor{myblue} & \cellcolor{myblue} & \cellcolor{myblue} & \cellcolor{myblue} & \cellcolor{myred} 2 & \cellcolor{myblue} & \cellcolor{myblue} & \cellcolor{myblue} & \cellcolor{myblue} & \cellcolor{myblue} & \cellcolor{myblue} & \cellcolor{myblue} & \cellcolor{myblue} & \cellcolor{myblue} & \cellcolor{myblue}\\ \hline
\rule{0pt}{0.5cm}19 & \cellcolor{myblue} & \cellcolor{myred} & \cellcolor{myblue} & \cellcolor{myblue} & \cellcolor{myblue} & \cellcolor{myblue} & \cellcolor{myblue} & \cellcolor{myblue} & \cellcolor{myblue} & \cellcolor{myblue} & \cellcolor{myblue} & \cellcolor{myblue} & \cellcolor{myblue} & \cellcolor{myblue} & \cellcolor{myblue} & \cellcolor{myblue} & \cellcolor{myblue} & \cellcolor{myblue} & \cellcolor{myblue} & \cellcolor{myblue}\\ \hline
\rule{0pt}{0.5cm}18 & \cellcolor{myblue} & \cellcolor{myred} & \cellcolor{myblue} & \cellcolor{myblue} & \cellcolor{myblue} & \cellcolor{myblue} & \cellcolor{myblue} & \cellcolor{myblue} & \cellcolor{myred} 2 & \cellcolor{myblue} & \cellcolor{myblue} & \cellcolor{myred} 3 & \cellcolor{myblue} & \cellcolor{myblue} & \cellcolor{myblue} & \cellcolor{myblue} & \cellcolor{myblue} & \cellcolor{myblue} & \cellcolor{myblue} & \cellcolor{myblue}\\ \hline
\rule{0pt}{0.5cm}17 & \cellcolor{myblue} & \cellcolor{myred} & \cellcolor{myblue} & \cellcolor{myblue} & \cellcolor{myblue} & \cellcolor{myblue} & \cellcolor{myblue} & \cellcolor{myblue} & \cellcolor{myblue} & \cellcolor{myblue} & \cellcolor{myblue} & \cellcolor{myblue} & \cellcolor{myblue} & \cellcolor{myblue} & \cellcolor{myblue} & \cellcolor{myblue} & \cellcolor{myblue} & \cellcolor{myblue} & \cellcolor{myblue} & \cellcolor{myblue}\\ \hline
\rule{0pt}{0.5cm}16 & \cellcolor{myblue} & \cellcolor{myred} & \cellcolor{myblue} & \cellcolor{myblue} & \cellcolor{myblue} & \cellcolor{myblue} & \cellcolor{myblue} & \cellcolor{myred} 2 & \cellcolor{myblue} & \cellcolor{myblue} & \cellcolor{myblue} & \cellcolor{myblue} & \cellcolor{myblue} & \cellcolor{myblue} & \cellcolor{myblue} & \cellcolor{myblue} & \cellcolor{myblue} & \cellcolor{myblue} & \cellcolor{myblue} & \cellcolor{myblue} \\ \hline
\rule{0pt}{0.5cm}15 & \cellcolor{myblue} & \cellcolor{myred} & \cellcolor{myblue} & \cellcolor{myblue} & \cellcolor{myblue} & \cellcolor{myblue} & \cellcolor{myblue} & \cellcolor{myblue} & \cellcolor{myblue} & \cellcolor{myred} 3 & \cellcolor{myblue} & \cellcolor{myblue} & \cellcolor{myblue} & \cellcolor{myblue} & \cellcolor{myblue} & \cellcolor{myblue} & \cellcolor{myblue} & \cellcolor{myblue} & \cellcolor{myblue} & \cellcolor{myblue} \\ \hline
\rule{0pt}{0.5cm}14 & \cellcolor{myblue} & \cellcolor{myred} & \cellcolor{myblue} & \cellcolor{myred} & \cellcolor{myblue} & \cellcolor{myblue} & \cellcolor{myred} 2 & \cellcolor{myblue} & \cellcolor{myblue} & \cellcolor{myblue} & \cellcolor{myblue} & \cellcolor{myblue} & \cellcolor{myblue} & \cellcolor{myblue} & \cellcolor{myblue} & \cellcolor{myblue} & \cellcolor{myblue} & \cellcolor{myblue} & \cellcolor{myblue} & \cellcolor{myblue}\\ \hline
\rule{0pt}{0.5cm}13 & \cellcolor{myblue} & \cellcolor{myred} & \cellcolor{myblue} & \cellcolor{myblue} & \cellcolor{myblue} & \cellcolor{myblue} & \cellcolor{myblue} & \cellcolor{myblue} & \cellcolor{myblue} & \cellcolor{myblue} & \cellcolor{myblue} & \cellcolor{myblue} & \cellcolor{myblue} & \cellcolor{myblue} & \cellcolor{myblue} & \cellcolor{myblue} & \cellcolor{myblue} & \cellcolor{myblue} & \cellcolor{myblue} & \cellcolor{myblue}\\ \hline
\rule{0pt}{0.5cm}12 & \cellcolor{myblue} & \cellcolor{myred} & \cellcolor{myblue} & \cellcolor{myblue} & \cellcolor{myblue} & \cellcolor{myred} 2 & \cellcolor{myblue} & \cellcolor{myred} 3 & \cellcolor{myblue} & \cellcolor{myblue} & \cellcolor{myblue} & \cellcolor{myblue} & \cellcolor{myblue} & \cellcolor{myblue} & \cellcolor{myblue} & \cellcolor{myblue} & \cellcolor{myblue} & \cellcolor{myred} 3 & \cellcolor{myblue} & \cellcolor{myblue} \\ \hline
\rule{0pt}{0.5cm}11 & \cellcolor{myblue} & \cellcolor{myred} & \cellcolor{myblue} & \cellcolor{myblue} & \cellcolor{myblue} & \cellcolor{myblue} & \cellcolor{myblue} & \cellcolor{myblue} & \cellcolor{myblue} & \cellcolor{myblue} & \cellcolor{myblue} & \cellcolor{myblue} & \cellcolor{myblue} & \cellcolor{myblue} & \cellcolor{myblue} & \cellcolor{myblue} & \cellcolor{myblue} & \cellcolor{myblue} & \cellcolor{myblue} & \cellcolor{myblue} \\ \hline
\rule{0pt}{0.5cm}10 & \cellcolor{myblue} & \cellcolor{myred} & \cellcolor{myred} & \cellcolor{myred} & \cellcolor{myred} 2 & \cellcolor{myblue} & \cellcolor{myblue} & \cellcolor{myblue} & \cellcolor{myblue} & \cellcolor{myblue} & \cellcolor{myblue} & \cellcolor{myblue} & \cellcolor{myblue} & \cellcolor{myblue} & \cellcolor{myred} 3 & \cellcolor{myblue} & \cellcolor{myblue} & \cellcolor{myblue} & \cellcolor{myblue} & \cellcolor{myred} 2 \\ \hline
\rule{0pt}{0.5cm}9 & \cellcolor{myblue} & \cellcolor{myred} & \cellcolor{myblue} & \cellcolor{myred} & \cellcolor{myblue} & \cellcolor{myred} 3 & \cellcolor{myblue} & \cellcolor{myblue} & \cellcolor{myblue} & \cellcolor{myblue} & \cellcolor{myblue} & \cellcolor{myblue} & \cellcolor{myblue} & \cellcolor{myblue} & \cellcolor{myblue} & \cellcolor{myblue} & \cellcolor{myblue} & \cellcolor{myred} 2 & \cellcolor{myblue} & \cellcolor{myblue} \\ \hline
\rule{0pt}{0.5cm}8 & \cellcolor{myblue} & \cellcolor{myred} & \cellcolor{myred} & \cellcolor{myred} 2 & ? & \cellcolor{myblue} & \cellcolor{myblue} & \cellcolor{myblue} & \cellcolor{myblue} & \cellcolor{myblue} & \cellcolor{myblue} & \cellcolor{myred} 3 & \cellcolor{myblue} & \cellcolor{myblue} & \cellcolor{myblue} & \cellcolor{myred} 2 & \cellcolor{myblue} & \cellcolor{myblue} & \cellcolor{myblue} & \cellcolor{myblue} \\ \hline
\rule{0pt}{0.5cm}7 & \cellcolor{myblue} & \cellcolor{myred} & \cellcolor{myred} & \cellcolor{myred} & \cellcolor{myblue} & \cellcolor{myred} & \cellcolor{myblue} & \cellcolor{myblue} & \cellcolor{myblue} & \cellcolor{myblue} & \cellcolor{myblue} & \cellcolor{myblue} & \cellcolor{myblue} & \cellcolor{myred} 2 & \cellcolor{myblue} & \cellcolor{myblue} & \cellcolor{myblue} & \cellcolor{myblue} & \cellcolor{myblue} & \cellcolor{myblue} \\ \hline
\rule{0pt}{0.5cm}6 & \cellcolor{myblue} & \cellcolor{myred} & \cellcolor{myred} 2 & \cellcolor{myred} 3 & \cellcolor{myblue} & \cellcolor{myblue} & \cellcolor{myred} & \cellcolor{myblue} & \cellcolor{myred} 3 & \cellcolor{myblue} & \cellcolor{myblue} & \cellcolor{myred} 2 & \cellcolor{myblue} & \cellcolor{myblue} & \cellcolor{myblue} & \cellcolor{myblue} & \cellcolor{myblue} & \cellcolor{myblue} & \cellcolor{myblue} & \cellcolor{myblue} \\ \hline
\rule{0pt}{0.5cm}5 & \cellcolor{myblue} & \cellcolor{myred} & \cellcolor{myred} & \cellcolor{myred} & \cellcolor{myblue} & \cellcolor{myblue} & \cellcolor{myblue} & ? & \cellcolor{myblue} & \cellcolor{myred} 2 & \cellcolor{myblue} & \cellcolor{myblue} &   \cellcolor{myblue} & \cellcolor{myblue} & \cellcolor{myblue} & \cellcolor{myblue} & \cellcolor{myblue} & \cellcolor{myblue} & \cellcolor{myblue} & \cellcolor{myblue} \\ \hline
\rule{0pt}{0.5cm}4 & \cellcolor{myblue} & \cellcolor{myred} & \cellcolor{myred} & \cellcolor{myblue} & \cellcolor{myred} & \cellcolor{myred} 3 & \cellcolor{myred} & \cellcolor{myred} 2 & \cellcolor{myred} & \cellcolor{myred} & \cellcolor{myblue} & \cellcolor{myblue} & \cellcolor{myblue} & \cellcolor{myred} & \cellcolor{myblue} & \cellcolor{myblue} & \cellcolor{myblue} & \cellcolor{myblue} & \cellcolor{myblue} & \cellcolor{myblue} \\ \hline
\rule{0pt}{0.5cm}3 & \cellcolor{myblue} & \cellcolor{myred} & \cellcolor{myblue} & \cellcolor{myred} & \cellcolor{myred} & \cellcolor{myred} 2 & \cellcolor{myred} & \cellcolor{myred} & \cellcolor{myblue} & \cellcolor{myred} & \cellcolor{myblue} & \cellcolor{myblue} & \cellcolor{myblue} & \cellcolor{myblue} & \cellcolor{myblue} & \cellcolor{myblue} & \cellcolor{myblue} & \cellcolor{myblue} & \cellcolor{myblue} & \cellcolor{myblue} \\ \hline
\rule{0pt}{0.5cm}2 & \cellcolor{myred} & \cellcolor{myred} & \cellcolor{myred} & \cellcolor{myred} & \cellcolor{myred} & \cellcolor{myred} & \cellcolor{myred} & \cellcolor{myred} & \cellcolor{myred} & \cellcolor{myred} & \cellcolor{myred} & \cellcolor{myred} & \cellcolor{myred} & \cellcolor{myred} & \cellcolor{myred} & \cellcolor{myred} & \cellcolor{myred} & \cellcolor{myred} & \cellcolor{myred} & \cellcolor{myred} \\ \hline
\rule{0pt}{0.5cm}1 & \cellcolor{myblue} & \cellcolor{myred} & \cellcolor{myblue} & \cellcolor{myblue} & \cellcolor{myblue} & \cellcolor{myblue} & \cellcolor{myblue} & \cellcolor{myblue} & \cellcolor{myblue} & \cellcolor{myblue} & \cellcolor{myblue} & \cellcolor{myblue} & \cellcolor{myblue} & \cellcolor{myblue} & \cellcolor{myblue} & \cellcolor{myblue} & \cellcolor{myblue} & \cellcolor{myblue} & \cellcolor{myblue} & \cellcolor{myblue} \\ \hline
\rule{0pt}{0.5cm}\slashbox{$p$}{$q$} &  1 & 2 & 3 & 4 & 5 & 6 & 7 & 8 & 9 & 10 & 11 & 12 & 13 & 14 & 15 & 16 & 17 & 18 & 19 & 20\\ \hline
\end{tabular}
}
\caption{Every blue cell corresponds to a positive pair. If $p\ge 3$ and $q\ge 3$, then we provide an explicit morphism constructing such words. Every red cell corresponds to a negative pair. A red cell with a $2$ means it lies on the $(t, 2t)$ line or on the $(2t, t)$ line and a red cell with a $3$ means it lies on the $(2t, 3t)$ line or on the $(3t, 2t)$ line. For the other red cells, the backtracking algorithm terminates.}
\label{table:pq_leq_20}
\end{center}
\end{table}

\section{Further remarks and open problems}\label{sec:further_remarks}

\paragraph{Non-coprime pairs}
In this work, we proved that for all but finitely many pairs of relatively prime integers $p$, $q \geq 3$, there exists an infinite ternary square-free word that is square-free modulo $p$ and $q$. In this setting, the two subsequences start at position $0$. For pairs of relatively prime integers, the starting positions do not matter as every relative difference between the two sequences will appear anyway. However, if we consider the set of all pairs of integers, then the starting positions of the subsequences matter. A problem more general than the one we studied is therefore the following.

\begin{question}
For which triples $(p, q, s)$ do there exist infinite ternary square-free words whose subsequences at positions congruent to $0$ modulo $p$ and to $s$ modulo $q$ are square-free?
\end{question}

In the simpler case $s = 0$, we already know that there are infinitely many negative pairs $(p, q)$, since the families $(t, 2t)_{t \geq 1}$ and $(2t, 3t)_{t \geq 1}$ consist entirely of negative pairs.
Moreover, in the proof of Theorem \ref{thm:psmall_qlarge}, we never use the hypothesis that $p$ and $q$ are relatively prime integers, so Theorem \ref{thm:psmall_qlarge} provides infinitely many positive pairs. In fact, in the proof of Theorem \ref{thm:psmall_qlarge} we do not need $s = 0$, so the same proof technique should provide infinitely many positive pairs for every possible value of $s$. 

\paragraph{k-uplets}

Another natural generalization of our problem is to consider more than two square-free subsequences.

\begin{question}
Let $k \geq 3$ be an integer. For which $k$-tuples of integers $(p_1, \dots, p_k)$ does there exist an infinite ternary square-free word that is square-free modulo $p_i$ for every $i \in \{1, \dots, k\}$?
\end{question}

In all our constructions, we start by choosing an infinite ternary square-free word $t$ that will be the subsequence modulo $q$ for some integer $q$ of an infinite ternary square-free word $w$. If $t$ is square-free modulo $q'$ as well, then $w$ will be square-free, square-free modulo $q$, and square-free modulo $qq'$. Thus, if we allow $p_1, \dots, p_k$ not to be pairwise relatively prime, then we can easily construct positive answers to this question by iterating our constructions. If we restrict ourselves to $k$-tuples of relatively prime integers, a more in-depth use of the technique from Theorem \ref{thm:pq_large} might work for small values of $k$.

A similar problem consists in considering an infinite number of subsequences. Let $(p_i)_{i \ge 0}$ be an increasing sequence of positive integers. Are there infinite ternary square-free words that are square-free modulo $p_i$ for every $i \geq 0$ ?
As previously observed, by chaining our constructions, we can construct such words. The difficulty seems to arise when we require that $p_i$ and $p_j$ be relatively prime for every $i > j \ge 0$.
\begin{question}
For which increasing sequences $(p_i)_{i \ge 0}$ of relatively prime integers do there exist infinite ternary square-free words $w$ such that for every $i \geq 0$, $\subsequence{w}{p_i}$ is square-free?
\end{question}

We already know that no such sequence exists with $p_0 = 2$. If there exists an increasing sequence $(p_i)_{i \ge 0}$ of relatively prime integers such that there is an infinite ternary square-free word that is square-free modulo $p_i$ for every $i$, what is the smallest possible value of $p_0$? Moreover, if the constraints are sparse, then the gaps between them should allow us to construct such words as we did in Section \ref{chap:pqlarge}.
\begin{question}
Do there exist integers $p$ and $k$ such that for every increasing sequence $(p_i)_{i \ge 0}$ of relatively prime integers with $p_0 = p$ and $p_{i+1} > kp_i$ for every $i \ge 0$, there exists an infinite ternary square-free word $w$ such that for every $i \geq 0$, $\subsequence{w}{p_i}$ is square-free?
\end{question}

\paragraph{All starting positions}
Square-free circular morphisms with images of size $p$ allow us to construct infinite ternary square-free words $w$ such that for every $\alpha \in \{0, \dots, p - 1\}$, the subsequence congruent to $\alpha$ modulo $p$ of $w$ is square-free. Such morphisms exist for every $p \in \{13, 17, 18, 19\} \cup \mathbb{N}_{\geq 23}$ by Proposition \ref{prop:currie2021infinite}. We can thus ask for the following strengthening.
\begin{question}
 For a given pair of relatively prime integers $(p, q)$, does there exist an infinite ternary square-free word $w$ such that for every $\alpha \in \{0, \dots, p - 1\}$, the subsequence congruent to $\alpha$ modulo $p$ of $w$ is square-free and the subsequence modulo $q$ of $w$ is square-free ?
\end{question}

The aforementioned morphisms together with Proposition \ref{prop:circular_morphism_sf_sfmodp} actually give positive answers to the previous question with the pairs $(p, q)$ with $p \in \{13, 17, 18, 19\} \cup \mathbb{N}_{\geq 23}$ and $q \ge 19p$.

\paragraph{Growth rate}

In Theorems \ref{thm:pq_large} and \ref{thm:psmall_qlarge}, the construction starts by choosing the word that will be the subsequence modulo $q$. It is well known that the language $\mathcal{L}$ of ternary square-free words is exponential in the sense that its complexity $C_\mathcal{L}(n) = f(n)\gamma^n$ for some sub-exponential function $f$ and some constant $\gamma$. Moreover, it is known that $$1.3017597< \lim\limits_{n \rightarrow \infty} C_\mathcal{L}(n)^{1/n} < 1.3017619$$  as given in \cite{SHUR2012187}. This means that if $\mathcal{L}_{p, q}$ is the language of ternary words that are square-free, square-free modulo $p$, and square-free modulo $q$ for a pair $(p, q)$ covered by Theorem \ref{thm:pq_large} or Theorem \ref{thm:psmall_qlarge}, then, $C_{\mathcal{L}_{p, q}}(n) \ge (\gamma^{1/q})^n$ and so, $\mathcal{L}_{p, q}$ is an exponential language. However, this bound is most likely not tight.
Consider instead the language $\mathcal{F}_{p, q}$ of factors of $\mathcal{L}_{p, q}$. By definition, this language is factor-closed and it is a classical consequence that $\lim\limits_{n \rightarrow \infty} C_{\mathcal{F}_{p, q}}(n)^{1/n}
= \inf\limits_{n}C_{\mathcal{F}_{p, q}}(n)^{1/n}$ is defined. Note that for all positive $p$, $q$, and $n$,
$\frac{C_{\mathcal{F}_{p, q}}(n)}{pq}\le C_{\mathcal{L}_{p, q}}(n)\le C_{\mathcal{F}_{p, q}}(n)$, hence the growth rate $\lim\limits_{n \rightarrow \infty} C_{\mathcal{L}_{p, q}}(n)^{1/n}=\lim\limits_{n \rightarrow \infty} C_{\mathcal{F}_{p, q}}(n)^{1/n}$ is also defined. This raises the following questions :

\begin{question}
Is it true that if there exists an infinite ternary square-free word that is square-free modulo $p$ and $q$, then $C_{\mathcal{L}_{p, q}}$ is an exponential language ? More precisely, determine the value of $\lim\limits_{n \rightarrow \infty} C_{\mathcal{L}_{p, q}}(n)^{1/n}$, for all $p$ and $q$. \end{question}

\bibliographystyle{plain}
\bibliography{biblio} 

@article{ArithmeticProgressions,
	author = {Currie, James and Harju, Tero and Ochem, Pascal and Rampersad, Narad},
	title = {Some further results on squarefree arithmetic progressions in infinite words},
	journal = {Theoret. Comput. Sci.},
	volume = {799},
	pages = {140--148},
	year = {2019},
	month = dec,
	issn = {0304-3975},
	publisher = {Elsevier},
	doi = {10.1016/j.tcs.2019.10.006}
}

@article{rosenfeld2020far,
  title={How far away must forced letters be so that squares are still avoidable?},
  author={Rosenfeld, Matthieu},
  journal={Mathematics of Computation},
  volume={89},
  number={326},
  pages={3057--3071},
  year={2020}
}

@article{currie2012infinite,
    title={Infinite ternary square-free words concatenated from permutations of a single word},
    author={Currie, James},
    journal = {Theoretical Computer Science},
    volume = {482},
    pages = {1-8},
    year = {2013},
}

@article{thue1906uber,
  title={Uber unendliche zeichenreihen},
  author={Thue, Axel},
  journal={Norske Vid Selsk. Skr. I Mat-Nat Kl.(Christiana)},
  volume={7},
  pages={1--22},
  year={1906}
}

@article{SHUR2012187,
title = {Growth properties of power-free languages},
journal = {Computer Science Review},
volume = {6},
number = {5},
pages = {187-208},
year = {2012},
issn = {1574-0137},
doi = {https://doi.org/10.1016/j.cosrev.2012.09.001},
url = {https://www.sciencedirect.com/science/article/pii/S1574013712000330},
author = {Shur, Arseny}
}

@article{HARJU201995,
title = {On square-free arithmetic progressions in infinite words},
journal = {Theoretical Computer Science},
volume = {770},
pages = {95-100},
year = {2019},
issn = {0304-3975},
doi = {https://doi.org/10.1016/j.tcs.2018.09.032},
url = {https://www.sciencedirect.com/science/article/pii/S0304397518306030},
author = {Harju, Tero}
}

@article{CROCHEMORE1982221,
title = {Sharp characterizations of squarefree morphisms},
journal = {Theoretical Computer Science},
volume = {18},
number = {2},
pages = {221-226},
year = {1982},
issn = {0304-3975},
doi = {https://doi.org/10.1016/0304-3975(82)90023-8},
url = {https://www.sciencedirect.com/science/article/pii/0304397582900238},
author = {Crochemore, Maxime}
}

@article{DEKKING2023102536,
title = {Two-block substitutions and morphic words},
journal = {Advances in Applied Mathematics},
volume = {148},
pages = {102536},
year = {2023},
issn = {0196-8858},
doi = {https://doi.org/10.1016/j.aam.2023.102536},
url = {https://www.sciencedirect.com/science/article/pii/S0196885823000544},
author = {Michel Dekking and Michael Keane}
}

@article{HARJU_palindromes,
title = {A note on short palindromes in square-free words},
journal = {Theoretical Computer Science},
volume = {562},
pages = {658-659},
year = {2015},
issn = {0304-3975},
doi = {https://doi.org/10.1016/j.tcs.2014.10.040},
url = {https://www.sciencedirect.com/science/article/pii/S0304397514008329},
author = {Tero Harju and Mike Müller}
}

@article{CURRIE_palindromes,
title = {Palindrome positions in ternary square-free words},
journal = {Theoretical Computer Science},
volume = {396},
number = {1},
pages = {254-257},
year = {2008},
issn = {0304-3975},
doi = {https://doi.org/10.1016/j.tcs.2007.09.015},
url = {https://www.sciencedirect.com/science/article/pii/S0304397507006883},
author = {James D. Currie}
}

@article{BRESAR_nonrepetitive_coloring_in_trees,
title = {Nonrepetitive colorings of trees},
journal = {Discrete Mathematics},
volume = {307},
number = {2},
pages = {163-172},
year = {2007},
issn = {0012-365X},
doi = {https://doi.org/10.1016/j.disc.2006.06.017},
url = {https://www.sciencedirect.com/science/article/pii/S0012365X06004778},
author = {B. Brešar and J. Grytczuk and S. Klavžar and S. Niwczyk and I. Peterin}
}

@misc{code,
  title = {Pairs of square free arithmetic progressions in infinite words https://github.com/ThomasDelepine/Pairs-of-square-free-arithmetic-progressions-in-infinite-words},
  author = {Thomas Del\'epine and Pascal Ochem and Matthieu Rosenfeld},
  year = {2026},
  howpublished = {GitHub repository},
  url = {https://github.com/ThomasDelepine/Pairs-of-square-free-arithmetic-progressions-in-infinite-words},
  urldate = {2026-01-27}
}

\appendix
\section{Appendix}
\begin{table}[H]
    \footnotesize
    \centering
    \begin{tabular}{ |C{0.25cm}|p{12cm}|C{0.25cm}|C{0.25cm}|p{0.8cm}| }
 \hline
 $p$ & $h(0)$ & $k$ & $\alpha$ & $q \ge$\\
 \hline
 $3$  & 012021201210201021012102012021012102120102012102120210 & $18$ & $1$ & $1080$\\
 \hline
 $4$  & 01202120102101201021201210120210201202101210201210120210121020120210 & $17$ & $0$ & $1360$\\
 \hline
 $5$  & 01210201021012021201020121012021012102010212021012102012021201210 & $13$ & $0$ & $1300$\\
 \hline
 $7$  & 01202102010210120102120121020102120102012021020121012021012102120102-
 10120102120121020120210 & $13$ & $0$ & $1820$\\
 \hline
 $8$  & 01202120102012102010212012101201020120212010201210120210201202120102-
        101210201021201210201021012102120210 & $13$ & $0$ & $2080$\\
 \hline
 $9$  & 01202102010210121020102120210120102101202120121012010210121020120210-
        1201020120210121020120212010210121021201020120210 & $13$ & $0$ & $2340$\\
 \hline
 $10$  & 01202120121020120210201021201210201202120102120210201021201210201202-
        10201021202102012021201020120210201021201210201202120102120210
 & $13$ & $0$ & $2600$\\
 \hline
 $11$  & 01202102010212021020121021201020120212012102010212010201210212021020-
         12021201210201202102012102120102012102010212012102120210201202120102-
         0120210
 & $13$ & $0$ & $2860$\\
 \hline
 $12$  & 01202120121012021012102120121020120210121021202101202120121020120210-
         20121012021012102012021201210120210121021201210120212012102012021201-
         21012021012102120210
 & $13$ & $0$ & $3120$\\
 \hline
 $14$  & 01202120121020120210121021202101202120121012021012102012021020121021-
         20210120212012101202101210212021020120210121020121012021201210201202-
         1201210120210121021201210201202102012102120210
 & $13$ & $0$ & $3640$\\
 \hline
 $15$  & 01210212012102012021020121012021012102120121012021012102012101202102-
         01210212012101202120121021202101210212012101202101210201202102012101-
         20210121020120212012102120210121020121012021020121021201210
 & $13$ & $0$ & $3900$\\
 \hline
 $16$  & 01210212012101201020121020102101201020121012010210121020121012010201-
         21020102120121012010210121021201020121012010210121020102101201020121-
         02010210121021201020121020102120121021201020121012010212012101201020-
         1210
 & $13$ & $0$ & $4160$\\
 \hline
 $20$  & 01210212012101202120121020120210121021202101202120121012021012102012-
         02102012101202101210201210120210201202101210212012101202101210201202-
         12012101202101210212012101202102012021012102120121012021012102012021-
         20121012021020121021201210120212012102120210121021201210
 & $13$ & $0$ & $5200$\\
 \hline
 $21$  & 01210212012101202101210201202102012101202120121020120210201210212012-
         10120210121020120210201210120210201202101210201202102012101202101210-
         20121012021201210201202101210212012101202101210201202102012101202101-
         210201202120121012021012102120121012021201210212021012102012021201210
 & $13$ & $0$ & $5460$\\
 \hline
 $22$  & 01210212012101202120121020120210121021201210201202102012101202101210-
         20120210201210212021012021201210120210121020120210201210120210121020-
         12021201210120210121020121012021020120210121020120212012101202101210-
         20120210201210212012101202101210201210120210201210212021012102120121-
         02012021201210
 & $13$ & $0$ & $5720$\\
 \hline
\end{tabular}
    \caption{Morphisms for Proposition \ref{prop:sparse_p_pairs}. Each line 
    contains: $p$, the morphism, the parameters $k$ and $\alpha$ of Proposition \ref{prop:sparse_p_pairs}, and the value of $q$ such that for all $q\geq Q$ there exists an infinite ternary square-free word that is square-free modulo $p$ and $q$.}
    \label{fig:morphisms_cone_remaining_cases}
\end{table}

\end{document}